\documentclass{article}
\usepackage[english]{babel}
\usepackage{graphicx}
\usepackage{amsfonts}
\usepackage{amsmath,amssymb}
\usepackage{amsthm}
\usepackage{float}
\usepackage{hyperref}
\usepackage{subfigure}
\usepackage[margin=1.5in]{geometry}
\usepackage{xcolor}

\title{Cameron-Liebler $k$-sets in $\text{AG}(n,q)$}
\begin{document}

\theoremstyle{plain}
\newtheorem{St}{Theorem}[section]
\newtheorem{theorem}[St]{Theorem}
\newtheorem{Le}[St]{Lemma}
\newtheorem{Gev}[St]{Corollary}
\theoremstyle{definition}
\newtheorem{Ex}[St]{Example}
\newtheorem{Def}[St]{Definition}
\newtheorem{Opm}[St]{Remark}
\newtheorem{Cons}[St]{Construction}
\newtheorem*{Consb}{Construction}

\newcommand{\AG}{\text{AG}}
\newcommand{\PG}{\text{PG}}
\newcommand{\PGL}{\text{PGL}}
\newcommand{\AGL}{\text{AGL}}
\newcommand{\cL}{\mathcal{L}}
\newcommand{\gauss}[2]{\genfrac{[}{]}{0pt}{}{#1}{#2}}
\renewcommand{\thefootnote}{\fnsymbol{footnote}}

 \footnotetext{$*$ Department of Mathematics: Analysis, Logic and Discrete Mathematics, Ghent University, Krijgslaan 281, Building S8, 9000 Gent, Flanders, Belgium (Email: jozefien.dhaeseleer@ugent.be, ferdinand.ihringer@ugent.be, leo.storme@ugent.be) \newline (http://cage.ugent.be/$\sim$ jmdhaese/, http://math.ihringer.org/, http://cage.ugent.be/$\sim$ls)\label{Gent}.}%
    \footnotetext{$\dagger$ Department of Mathematics and Data Science, University of Brussels (VUB),  Pleinlaan 2, Building G, 1050 Elsene, Flanders, Belgium  (Email: Jonathan.Mannaert@vub.be)\newline(http://homepages.vub.ac.be/$\sim$jonmanna/)\label{VUB}.}%

\author{J. D'haeseleer\footnotemark\label{Gent}, F.Ihringer$^*$, J. Mannaert\footnotemark{\label{VUB}} and L. Storme$^*$}

\maketitle

\begin{abstract}
        We study Cameron-Liebler $k$-sets in the affine geometry, so sets of $k$-spaces in $\AG(n, q)$. This generalizes research on Cameron-Liebler $k$-sets in the projective geometry $\PG(n, q)$.
        Note that in algebraic combinatorics, Cameron-Liebler $k$-sets of $\AG(n, q)$ correspond to certain equitable bipartitions of the association
        scheme of $k$-spaces in $\AG(n, q)$, while in the analysis of Boolean functions, they correspond to Boolean degree $1$ functions of $\AG(n, q)$. 
        
        We define Cameron-Liebler $k$-sets in $\AG(n, q)$ { in a similar way as Cameron-Liebler $k$-sets in $\PG(n,q)$, such that its characteristic vector is a linear combination of point-pencils.}
        In particular, we investigate the relationship between Cameron-Liebler $k$-sets in $\AG(n, q)$ and $\PG(n, q)$.
        As a by-product, we calculate the character table of the association scheme of affine lines.
        Furthermore, we characterize the smallest examples of Cameron-Liebler $k$-sets.
        
        This paper focuses on $\AG(n, q)$ for $n > 3$, while the case for Cameron-Liebler line classes in $\AG(3, q)$ was already treated separately.
\end{abstract}

\section{Introduction}

The investigation of Cameron-Liebler line classes in the projective geometry $\PG(n, q)$ 
goes back to Cameron and Liebler in 1982 \cite{Cameron-Liebler}. 
Their motivation was the investigation of the subgroup structure of $\PGL(n+1, q)$.
Particularly, a line orbit of a subgroup of $\PGL(n+1, q)$ acting on $\PG(n, q)$ with the same number of point- and line-orbits is 
a Cameron-Liebler line class. 

The concept of Cameron-Liebler line classes was rediscovered several times, see the introduction of \cite{Delsarte77} for a short overview.
In particular, algebraic combinatorialists studied equitable bipartitions as a natural generalization of perfect codes 
under various names in several highly symmetric families of graphs such as hypercubes and Johnson graphs. Similarly,
they also correspond to Boolean degree 1 functions in the analysis of Boolean functions.

In the special case of $\PG(3, q)$, a Cameron-Liebler line class can be defined as 
a family of lines which intersects all spreads of $\PG(3, q)$ in exactly $x$ lines for some constant $x$ \cite{Cameron-Liebler}.
We call $x$ the \textit{parameter} of the Cameron-Liebler line class.
In $\PG(3, q)$, there exists a list of examples which we refer to as \textit{trivial}: (1) the empty set
with parameter $x=0$, (2) all lines through a fixed point with parameter $x=1$, 
(3) all lines in a fixed plane with parameter $x=1$, (4) the union of (2) and (3), when disjoint, with parameter $x=2$,
and (5)--(8) the complements of (1)--(4) with parameters $x=q^2+1, q^2, q^2, q^2-1$.
Cameron and Liebler conjectured that these are the only examples. This was disproven by Drudge who found an example with parameter $x=5$ in PG($3,3$) \cite{Drudge}. Nowadays there are several infinite families of non-trivial examples known \cite{BruenAndDrude,DeBeule2016,Feng2015,Gavrilyuk2018}.
In contrast to this, there are no non-trivial examples known for $n > 3$. Hence,
there is some difference in behaviour between $n=3$ and $n > 3$.
This carries over to $\AG(n, q)$, where this paper handles the case $n>3$, while we treat the case $n=3$ separately in \cite{DMSS}.

Cameron-Liebler line classes were generalized to $k$-spaces of $\PG(n, q)$ 
in \cite{DBD2018,Delsarte77}. These are families of $k$-spaces which lie in the span of the point-($k$-space) incidence matrix.
We call such families \textit{Cameron-Liebler $k$-sets of $\PG(n, q)$}.
Note that if $\cL$ is a Cameron-Liebler $k$-set of $\PG(n, q)$, then its parameter $x$ is defined by $|\cL| = x \gauss{n}{k}_q$ where
$\gauss{n}{k}_q$ denotes the  Gaussian binomial coefficient.
Analogously, we call a family $\mathcal{L}$ of $k$-spaces which lies in the span of the (affine) point-($k$-space) incidence matrix  a \emph{Cameron-Liebler $k$-set of $\AG(n, q)$}. Moreover, if $|\cL| = x \gauss{n}{k}_q$, we say that $x$ is the parameter of the Cameron-Liebler $k$-set $\cL$.

After some preliminaries, we start our paper with some general properties of Cameron-Liebler $k$-sets in Section \ref{sec:AGgen}.
In particular, we show the following.

\begin{St}\label{H6ProjToAff}
Let $\cL$ be a Cameron-Liebler $k$-set with parameter $x$ in PG($n,q$) which does not contain $k$-spaces in some hyperplane $H$. 
Then $\cL$ is a Cameron-Liebler $k$-set with parameter $x$ of $\AG(n, q) \cong \PG(n, q) \setminus H$.
\end{St}

%An immediate consequence is that the following results for Cameron-Liebler $k$-sets 
%of $\PG(n, q)$ also hold for $\AG(n, q)$.
%
%%\begin{theorem}[{\cite[Theorem 4.9]{DBD2018}}]
% % Let $n \geq 3k+2 > 5$. There are no non-trivial Cameron-Liebler $k$-sets with parameter $x$ in $\AG(n, q)$ for $2 \leq x \leq C q^{\frac{n}{4} - k - D}$
% % for universal constants $C$ and $D$.
%%\end{theorem}
%%The statement in \cite[Theorem 4.9]{DBD2018} is more precise, but the proof contains an error and at the 
%%time of writing the upcoming corrigendum is still in preparation. The version stated here is correct and represents the
%%spirit of the bound in \cite{DBD2018} correctly.
%For $n = 2k+1$, there is a better bound.
%
%\begin{theorem}[Metsch {\cite[Theorem 1.4]{Metsch2017}}]
% For $k\geq 3$, there are no non-trivial Cameron-Liebler $k$-sets with parameter $x$ in $\AG(2k+1, q)$ for $2 \leq x \leq q/5$ and $q \geq q_0$
% for some universal constant $q_0$.
%\end{theorem}
%
%Note that \cite[Theorem 1.4]{Metsch2017} requires that $k < q \log q - q - 1$. This condition can be removed, see \cite[Theorem 1.8]{Ihringer2019}. Complementary to the two previous results, Theorem \ref{H6ProjToAff} also implies that the situation is known for small $q$.
%
%\begin{theorem}[{\cite[Theorem 1.4]{Delsarte77}}]
%  Let $n \geq 2k+1 > 3$. There are no non-trivial Cameron-Liebler $k$-sets in $\AG(n, q)$ for $q \leq 5$.
%\end{theorem}
We also obtained a result on the converse of Theorem \ref{H6ProjToAff}. In particular, we show the following.

\begin{St}\label{CLkSpaceBasic_plus_lines}\label{CLkSpaceBasic}
   % Let $(k+1) \,|\, (n+1)$ or $k = 1$.
    If $\cL$ is a Cameron-Liebler $k$-set of $\AG(n, q)$ with parameter $x$, then $\cL$ is a Cameron-Liebler $k$-set of $\PG(n, q)$
    with parameter $x$ in the projective closure $\PG(n, q)$ of $\AG(n, q)$.
\end{St}

%To show this theorem for $k = 1$, we make use of the character table of the association scheme of affine lines.
 We prove this theorem in general, but for the case $k=1$ we give an alternative proof using the character table of the association scheme of affine lines.
This association scheme has been investigated before as it is a well-known 3-class association scheme, see \cite{VanDam2000} for a more detailed study.
We could not find the character table of the affine lines scheme in the literature, so we provide the latter in Section \ref{sec:assoc}. While for $\PG(n, q)$ the character tables of the association scheme 
of $k$-spaces is explicitly known due to Delsarte \cite{Delsarte1976} and Eisfeld \cite{Eisfeld1999}, the determination of the character tables of the association scheme of $k$-spaces in $\AG(n, q)$
is still open.

A 3-class association scheme has four common eigenspaces $V_0, V_1, V_2, V_3$, where $V_0$ is spanned
by the all-ones vector. In our ordering, we provide explicit bases for $V_0 + V_1$ and $V_0 + V_3$, and we give a spanning set for $V_0 + V_2 + V_3$.

An immediate consequence of Theorem \ref{CLkSpaceBasic_plus_lines} is that the following results for Cameron-Liebler $k$-sets 
of $\PG(n, q)$ are also valid for Cameron-Liebler $k$-sets of $\AG(n, q)$.%, with the assumption that $(k+1)\mid(n+1)$ or $k=1$. More general, we call the  assumption that every Cameron-Liebler $k$-set in $\AG(n,q)$ is a Cameron-Liebler $k$-set  of the same parameter $x$ in $\PG(n,q)$ by $S_1$. In particular $S_1$ holds if $(k+1)\mid(n+1)$ or $k=1$.

%\begin{St}\cite[Theorem 4.9]{DBD2018}\label{JozefienClassific}
%	Let $A_1$ be true, then there are no Cameron-Liebler $k$-sets in $\AG(n,q)$, with $n\geq 3k+2$, of parameter 
%	$$ 2 \leq x \leq q^{\frac{n}{2}-\frac{k^2}{4}-\frac{3k}{4}-\frac32}(q-1)^{\frac{k^2}{4}-\frac{k}{4}+\frac12} \sqrt{q^2+q+1}.$$
%\end{St}

\begin{St}\cite[Theorem 4.9]{DBD2018}\label{JozefienClassific}
	%Let $S_1$ be true, then t
There are no Cameron-Liebler $k$-sets in $\AG(n,q)$, with $n\geq 3k+2$ and $q\geq 3$, of parameter 
	$$2\leq x\leq \frac{1}{\sqrt[8]{2}}q^{\frac{n}{2}-\frac{k^2}{4}-\frac{3k}{4}-\frac{3}{2}}(q-1)^{\frac{k^2}{4}-\frac{k}{4}+\frac{1}{2}}\sqrt{q^2+q+1}.$$
\end{St}
 The statement in \cite[Theorem 4.9]{DBD2018} is slightly different and its proof currently contains an error. 
The authors of \cite{DBD2018} have submitted an erratum. The version in this article is based on the recent corrected version on arXiv.\footnote{See arXiv:1805.09539v3 [math.CO]}
For $n \geq \frac52 k + \frac32$, a similar bound is given in \cite[Theorem 7]{Ihringer2020}.
For $n = 2k+1$, there is a better bound.

\begin{theorem}\cite[(Metsch, Theorem 1.4)]{Metsch2017}
 For $k\geq 3$, %and $S_1$,
 there are no non-trivial Cameron-Liebler $k$-sets with parameter $x$ in $\AG(2k+1, q)$ for $2 \leq x \leq q/5$ and $q \geq q_0$
 for some universal constant $q_0$.
\end{theorem}

Note that \cite[Theorem 1.4]{Metsch2017} requires that $k < q \log q - q - 1$. This condition can be removed, see \cite[Theorem 1.8]{Ihringer2019}. Complementary to the two previous results, Theorem \ref{H6ProjToAff} also implies that the situation is known for small $q$.

\begin{theorem}{\cite[Theorem 1.4]{BDF}}
  Let $n \geq 2k+1 > 3$. %and suppose that $S_1$ holds. 
Then there are no non-trivial Cameron-Liebler $k$-sets in $\AG(n, q)$ for $q \leq 5$.
\end{theorem}

We conclude with Section \ref{sec:class}, where we obtain a classification of the smallest Cameron-Liebler $k$-sets of $\AG(n, q)$.

\begin{theorem}\label{Smallx}
  All Cameron-Liebler $k$-sets of $\AG(n, q)$ with parameter $x \leq 2$ are trivial.
\end{theorem}

Note that this cannot be deduced from the literature on $\PG(n, q)$. We are also able to classify all Cameron-Liebler sets of hyperplanes in AG($n,q$), this will be done in Section \ref{sec:hyp}. We conclude with suggestions for future work in Section \ref{sec:fut}.

	\section{Preliminaries}
	Consider a prime $p$ and let $q=p^h$, with $h\geq 1$. Then consider $\PG(n,q)$, and $\AG(n,q)$, for $n>2$, as the $n$-dimensional projective, and affine, space over $\mathbb{F}_q$ respectively. % Notice that these two spaces are linked.
 Suppose that we consider a hyperplane $\pi_\infty$ in $\PG(n,q)$, which from now on will be called the hyperplane at infinity. Then we can consider all the points, lines, planes, and other spaces that do not lie inside $\pi_\infty$. In this way we obtain the affine space $\AG(n,q)$. %This observation will be of great use. 
The following notation will be used throughout this article.
\begin{Def}
 For $a,b \in \mathbb{N}$, we denote the \emph{Gaussian binomial coefficient} by
  $$\gauss{a}{b}_q= \frac{(q^a-1)\cdots(q^{a-b+1}-1)}{(q^b-1)\cdots (q-1)}.$$
  The Gaussian binomial coefficient $\gauss{a}{b}_q$ equals the number of $(b-1)$-spaces in  $\PG(a-1,q)$. Here we define that  $\gauss{a}{b}_q=0$ if $b>a$.
\end{Def}

 In general we take $k\geq 1$ with $n\geq k+1$, unless otherwise stated. Note that if we also ask that $(k+1)\mid (n+1)$, then it automatically follows that $n\geq 2k+1$. % or $n=k$. here the last possibility is not interesting and thus will be ruled out.

\begin{Def}
Consider $\PG(n,q)$, or $\AG(n,q)$ respectively.
\begin{enumerate}
\item A \emph{partial $k$-spread} is a set of pairwise disjoint $k$-spaces.
\item A \emph{conjugated switching $k$-set} is a pair of disjoint partial $k$-spreads that cover the same set of points.
\item A \emph{$k$-spread} is a partial $k$-spread that partitions the point set of $\PG(n,q)$, or $\AG(n,q)$ respectively.
\end{enumerate}
\end{Def}

\begin{Opm}
Since there is a lot of interaction between $\PG(n,q)$ and $\AG(n,q)$ we want to clarify some formulations. 
\begin{itemize}
\item Suppose that we have a set of $k$-spaces $\mathcal{L}$ in $\PG(n,q)$. Then the \emph{restriction} of $\mathcal{L}$ to $\AG(n,q)$ is the set of $k$-spaces of $\mathcal{L}$ that are not contained in $\pi_\infty$.
\item We say that two $k$-spaces  are \emph{projectively (or affinely) disjoint} if they intersect in the empty set in $\PG(n,q)$ (or $\AG(n,q)$ respectively).
\end{itemize}
\end{Opm}

We will give some examples of $k$-spreads in $\AG(n,q)$, which we will denote by respectively type I, II and III for future purposes.
\begin{Le}\label{SpreadsAG(n,q)}
			Consider the affine space $\AG(n,q)$ and the corresponding projective space $\PG(n,q)$. Then the following $k$-sets $\mathcal{S}$ are  $k$-spreads in $\AG(n,q)$.
			
			\begin{enumerate}
				\item (\emph{Type I}) %The set of affine $k$-spaces inside an arbitrary $k$-spread of the projective closure $\PG(n,q)$.
Every  $k$-spread in $\PG(n,q)$ restricted to the affine space.
				\item (\emph{Type II}) Consider a $(k-1)$-space $K$ in \(\pi_\infty\) and define	the set \(\mathcal{S}\) as the set of all affine $k$-spaces through $K$.
				\item(\emph{Type III}) 	Consider an $(n-2)$-space \(\pi_{n-2}\) in \(\pi_\infty\), then there are exactly \(q\) other hyperplanes through $\pi_{n-2}$ not equal to \(\pi_\infty\). Call these hyperplanes \(\pi_i\), for \(i \in \{1,...,q\}\). If we select for every hyperplane \(\pi_i\) a $(k-1)$-space $\tau_i\subseteq \pi_{n-2}$ (not all equal), then we can define the $k$-spread
				 \[\mathcal{S}:=\{K \mid K \text{ a $k$-space in }\AG(n,q), \,\,\, \tau_i \subseteq K\subseteq \pi_i \text{ for some } i \in \{1,...,q\}\}.\]		
			 \end{enumerate}
	\end{Le}
 \begin{proof}
 	\begin{enumerate}
 		\item  Consider a projective $k$-spread $\mathcal{S}$, then it is clear that every two $k$-spaces of $\mathcal{S}$ are affinely disjoint.   Secondly, $\mathcal{S}$ restricted to $\AG(n,q)$ partitions the point set of $\AG(n,q)$, since its extension reaches every  (affine) point in $\PG(n,q)$.
 		\item Trivial.
 		\item It is clear that all these elements are disjoint. Thus we only need to prove that for every affine point \(p\) there exists an element of \(\mathcal{S}\) that contains it. Consider for this point \(p\) the hyperplane $\langle p, \pi_{n-2} \rangle$, then this is a hyperplane through \(\pi_{n-2}\). Without loss of generality we may assume that it is \(\pi_i\). Such that $\langle p, \tau_i \rangle$ is a $k$-space in \(\mathcal{S}\) which contains \(p\). This proves that \(\mathcal{S}\) is indeed a  $k$-spread.
 	\end{enumerate}

\end{proof}

\begin{Opm}
The size of a $k$-spread in $\AG(n,q)$ is equal to $\frac{q^n}{q^k}=q^{n-k}$, where $q^k$ is the number of points in an affine $k$-space and $q^n$ is the total number of points in $\AG(n,q)$. An analogous result can be obtained in $\PG(n,q)$, where the size of a $k$-spread is known to be $\frac{q^{n+1}-1}{q^{k+1}-1}$. Note that this number is only an integer if $(k+1)\mid (n+1)$, so it follows that this is a necessary condition for the existence of $k$-spreads in $\PG(n,q)$. It is proven in \cite[Corollary 4.17]{Hirschfeld} that this is also a sufficient condition. 
\end{Opm}

\begin{Def}
Let us denote the set of $k$-spaces in $\PG(n,q)$, and $\AG(n,q)$, by $\Pi_k$, and $\Phi_k$, respectively. If we number the points and the $k$-spaces in these spaces, then we can define the point-($k$-space) incidence matrix $P_n$ and $A_n$ respectively. These matrices are $0,1$-valued matrices with a $1$ on position $(i,j)$ if and only if point $i$ lies on $k$-space $j$.
\end{Def}
We now give a special construction for the matrix $P_n$.
\begin{Cons}[Incidence matrix]\label{IncidenceMatrix}
Consider now the point-($k$-space) incidence matrix \(P_n\) of $\PG(n,q)$, where the rows correspond to the points and the columns correspond to the elements of \(\Pi_k\). We order the rows and columns in such a way that the first rows and columns correspond to the affine points and  affine $k$-spaces respectively.  Then \(P_n\) is of the following form:
\begin{equation}\label{IncidenceMatrixPn}
P_n= \begin{bmatrix}
A_n & \bar{0} \\
B_2 & P_{n-1}
\end{bmatrix},
\end{equation}
where \(A_n\) is the incidence matrix of $\AG(n,q)$, where again the rows  correspond to the points and the columns correspond to the elements of \(\Phi_k\). The matrix $\bar{0}$ is the zero-matrix and the part that remains unnamed, we call \(B_2\).
\end{Cons}
%In the following results $A$, or $P_n$, is the point-($k$-space) incidence matrix of AG($n,q$), or PG($n,q$) respectively.
We will use the notation of Construction \ref{IncidenceMatrix} in the following results. These results give some information about the characteristic vector of Cameron-Liebler $k$-sets. This characteristic vector is a $0,1$ valued vector, which contains a $1$ on position $i$ if and only if the $i$th $k$-space belongs to the Cameron-Liebler $k$-set.

\begin{Le}\cite[Lemma 2.4]{DMSS}\label{2to8}
	Consider a set \(\mathcal{L}\) of $k$-spaces in $\AG(n,q)$ such that \(\chi_{\mathcal{L}} \in (\ker(A_n))^\perp=\text{Im}(A_n^T)\). 
Then it follows for every affine $k$-spread \(\mathcal{S}\) that
	\[|\mathcal{L}\cap\mathcal{S}|=x,\]
	for a fixed integer \(x\).
\end{Le}

\begin{Opm}
We should also note that in general it holds that $\text{Im}(P_n^T)=(\ker(P_n))^\perp$ and $\text{Im}(A_n^T)=(\ker(A_n))^\perp$.
\end{Opm}
\begin{St}\cite[Theorem 2.3]{DMSS}\label{chiUpToDown}
	Consider the projective space $\PG(n,q)$ and consider a set of $k$-spaces \(\mathcal{L}\). If its characteristic vector \(\chi_{\mathcal{L} }\in (\ker(P_n))^\perp \) and \(\mathcal{L}\) also contains no $k$-spaces at infinity, then  \(\chi_{\mathcal{L} } \) restricted to the affine space belongs to \( (\ker(A_n))^\perp\). 
\end{St}

\subsection{Cameron-Liebler $k$-sets in $\PG(n,q)$}
Our goal here is to state some important results that are known for Cameron-Liebler $k$-sets in $\PG(n,q)$.
We start with the definition of Cameron-Liebler $k$-sets in $\PG(n,q)$. %Since not every space $\PG(n,q)$ has $k$-spreads, we cannot define these sets by $k$-spreads.
\begin{Def}

 A Cameron-Liebler $k$-set $\mathcal{L}$ in $\PG(n,q)$ is a set of $k$-spaces such that for its characteristic vector $\chi_\cL$, it holds that $\chi_\cL \in$ Im($P_n^T$). We say that $\mathcal{L}$ has parameter $x:=|\mathcal{L}|/\gauss{n}{k}$.% if and only if $$|\mathcal{L}|=x \gauss{n}{k}_q.$$

\end{Def}
\begin{Opm}
The fact that $\chi_\cL \in$ Im($P_n^T$) states that $\chi_\cL$ is a linear combination of the rows of $P_n^T$. In some literature, for example \cite{BDF},  the characteristic vector $\chi_\cL$ is called a Boolean degree 1 function in $\PG(n,q)$. Similarly we can consider $\chi_\cL \in$ Im($A_n^T$) for the incidence matrix $A_n$ in $\AG(n,q)$.
\end{Opm}

In this section we list some results on Cameron-Liebler $k$-spaces in PG($n,q$). We refer to \cite{DBD2018} for more information. 
 %Most of the results we mentioned are stated in \cite{DBD2018}. In this article Cameron-Liebler $k$-spaces where studied in PG($n,q$).

\begin{St}\cite[Theorem 2.9]{DBD2018}\label{EquivalenceProj}
	Let \(\mathcal{L}\) be a non-empty set of $k$-spaces in $\PG(n,q)$, $n \geq 2k+1$, with characteristic vector $\chi$, and $x$ so that 
	$|\mathcal{L}|= x\gauss{n}{k}_q$. Then the following properties are equivalent.
	\begin{enumerate}
		\item The set $\mathcal{L}$ is a Cameron-Liebler $k$-set. % $\chi \in $ Im($P_n^T$)$=(\ker(P_n))^\perp$, with $P_n$ the point-($k$-space) incidence matrix of $\PG(n,q)$.
		\item For every $k$-space $K$, the number of elements of $\mathcal{L}$ disjoint from $K$ is equal to $(x-\chi(K))\gauss{n-k-1}{k}_q q^{k^2+k}$.
		%\item For an $i \in \{1,...,k+1\}$ and a given $k$-space $K$, the number of elements of $\mathcal{L}$, meeting $K$ in a $(k-i)$-space is given by:
		%$$\left\{\begin{array}{ll}
		%	  \left( (x-1) \frac{q^{k+1}-1}{q^{k-i+1}-1}+ q^i \frac{q^{n-k}-1}{q^i-1}\right) q^{i(i-1)} \begin{bmatrix}
		%	n-k-1 \\
		%	i-1
		%	\end{bmatrix}_q \begin{bmatrix}
		%	k \\
		%	i
		%	\end{bmatrix}_q & \text{ if } K \in \mathcal{L} \\
		%	xq^{i(i-1)} \begin{bmatrix}
		%	n-k-1 \\
		%	i-1
		%	\end{bmatrix}_q \begin{bmatrix}
		%	k+1 \\
		%	i
		%	\end{bmatrix}_q & \text{ if } K \not\in \mathcal{L}.
		%\end{array} \right.$$
		 \item For every pair of conjugated switching sets $\mathcal{R}$ and $\mathcal{R}'$, $|\mathcal{L}\cap \mathcal{R}|=|\mathcal{L}\cap \mathcal{R}'|$.

	\end{enumerate}
If $\PG(n,q)$ has a $k$-spread, then the following property is equivalent to the previous ones.
\begin{enumerate}
\item[4] For every $k$-spread $\mathcal{S}$, $|\mathcal{L}\cap \mathcal{S}|=x$.
\end{enumerate}
\end{St}

\begin{Ex}\cite[Example 3.2]{DBD2018}\label{TrivialEx}
The following $k$-sets are examples of Cameron-Liebler $k$-sets in $\PG(n,q)$.
\begin{enumerate}
 \item The set of all the $k$-spaces through a fixed point is an example of a Cameron-Liebler $k$-set of parameter $x=1$.

 \item If we consider the set of $k$-spaces inside a fixed hyperplane, then this is a Cameron-Liebler $k$-set of parameter $x=\frac{q^{(n-k)}-1}{q^{(k+1)}-1}$. Note that $x$ is only an integer if and only if $(k+1)\mid (n+1)$.
\end{enumerate}
\end{Ex}
In order to give some context on the study of Cameron-Liebler $k$-sets in $\PG(n,q)$, we now give some classification results
\begin{St}\cite[Theorem 4.1]{BDF}\label{Non-ExBD1}
Let $q \in\{2,3,4,5\}$. Then all Cameron-Liebler $k$-sets in $\PG(n,q)$ are of the form of Example \ref{TrivialEx}, if $k,n-k \geq 2$ and either (a) $n \geq 5$ or (b) $n = 4$ and $q = 2$.
\end{St}

\begin{St}\cite[Theorem 4.1]{DBD2018} \label{CharX=1PG(n,q)}
Let $\mathcal{L}$ be a Cameron-Liebler $k$-set  with parameter $x = 1$ in $\PG(n,q)$, $n \geq 2k+1$. Then $\mathcal{L}$ consists out of all the $k$-spaces through a fixed point or $n = 2k+1$ and $\mathcal{L}$ is the set of all the $k$-spaces in a hyperplane of PG($2k + 1,q$).
\end{St}

\begin{St}\cite[Theorem 4.2]{DBD2018}
 There are no Cameron–Liebler $k$-sets in PG($n,q$) with parameter $x \in ] 0,1[$ and if $n \geq 3k + 2$, then there are no Cameron–Liebler $k$-sets with parameter $x\in ]1,2[$.
\end{St}
%\begin{St}(\cite[Theorem 4.9]{DBD2018})\label{JozefienClassific}
%	There are no Cameron-Liebler $k$-sets in PG($n,q$), with $n\geq 3k+2$, of parameter 
%	$$ 2 \leq x \leq q^{\frac{n}{2}-\frac{k^2}{4}-\frac{3k}{4}-\frac32}(q-1)^{\frac{k^2}{4}-\frac{k}{4}+\frac12} \sqrt{q^2+q+1}.$$
%\end{St}

\section{Cameron-Liebler $k$-sets in $\AG(n,q)$}\label{sec:AGgen}

 \begin{Def}
A Cameron-Liebler $k$-set $\mathcal{L}$ in $\AG(n,q)$  is a set of $k$-spaces such that for its characteristic vector $\chi_{\mathcal{L}}$, it holds that  $\chi_{\mathcal{L}}\in \text{Im}(A_n^T)$. Here we say that $\mathcal{L}$ has parameter $x:=|\mathcal{L}|/\gauss{n}{k}$.
\end{Def}

Due to Lemma \ref{SpreadsAG(n,q)}, we know that for every value of $n$, the affine space $\AG(n,q)$ contains $k$-spreads. %Therefore this definition is unambiguous. Remark that this definition implies that $x$ is always an integer, this is an important difference with Cameron-Liebler $k$-spaces in $\PG(n,q)$.
By Lemma \ref{2to8}, we find that Cameron-Liebler $k$-sets have a constant intersection number with $k$-spreads. This number will be equal to the parameter $x$ of the Cameron-Liebler $k$-set. This is proven in the next lemma.

\begin{Le}\label{SizeAffCLkSet}\label{ConstInt}
Suppose that $\mathcal{L}$ is a Cameron-Liebler $k$-set in $\AG(n,q)$, then it holds for every $k$-spread $\mathcal{S}$ that
$$|\mathcal{L}\cap \mathcal{S}|=x,$$
where $x$ is the parameter of $\mathcal{L}$.
%	Consider a Cameron-Liebler $k$-set $\mathcal{L}$ in $\AG(n,q)$ with parameter $x$, then
%	$$|\mathcal{L}|= x \begin{bmatrix}
%	n \\
%	k
%	\end{bmatrix}_q.$$
\end{Le}
\begin{proof}
 Due to Lemma \ref{2to8}, it follows that for every $k$-spread $\mathcal{S}$ it holds that $|\cL\cap \mathcal{S}|=c$, for a certain fixed integer $c$. To prove that $c=x$, we double count the pairs  $(K,\mathcal{S})$, where $\mathcal{S}$ is a $k$-spread of type II and $K \in \mathcal{L}\cap \mathcal{S}$. So we obtain that
$$x\gauss{n}{k}_q\cdot 1=\gauss{n}{k}_q\cdot c.$$
This proves the lemma.
% The first part follows from Lemma \ref{2to8} and the fact that the parameter is equal to $x$ follows from double counting the pairs $(K,\mathcal{S})$, where $\mathcal{S}$ is a $k$-spread of type II and $K \in \mathcal{L}\cap \mathcal{S}$.
\end{proof}
Remark that this statement holds for every $n$ and $k$, since in each case we have these $k$-spreads. We now give some basic properties.
\begin{Le}\label{CLkSetProp}\label{Chapter6Porperties}
	Consider \(\mathcal{L}\) and \(\mathcal{L}'\) to be Cameron-Liebler $k$-sets with parameter $x$ and $x'$ both in $\AG(n,q)$ or both in $\PG(n,q)$ respectively, then the following properties hold.
	\begin{enumerate}
		\item If $\cL$ is a Cameron-Liebler $k$-set in $\AG(n,q)$ or $\PG(n,q)$, then we have that $0 \leq x\leq q^{n-k}$ or $0 \leq x\leq \frac{q^{n+1}-1}{q^{k+1}-1}$ respectively.
		\item If \(\mathcal{L} \cap \mathcal{L}'= \emptyset\), then \(\mathcal{L}\cup \mathcal{L}'\) is a Cameron-Liebler $k$-set of parameter $x+x'$.
		\item If \(\mathcal{L}' \subseteq \mathcal{L}\), then \(\mathcal{L}\setminus \mathcal{L}'\) is a Cameron-Liebler $k$-set of parameter $x-x'$.
		\item If $\cL$ is a Cameron-Liebler $k$-set in $\AG(n,q)$ or $\PG(n,q)$, then the complement of $\mathcal{L}$ in $\AG(n,q)$ or $\PG(n,q)$ is a Cameron-Liebler $k$-set with parameter \(q^{n-k}-x\) or $\frac{q^{n+1}-1}{q^{k+1}-1}-x$ respectively.
	\end{enumerate}
\end{Le}
\begin{proof}
	This lemma follows due to Lemma \ref{SizeAffCLkSet} and \cite[Lemma 3.1]{DBD2018} for the projective case.
\end{proof}

We now give some general results that will give connections between Cameron-Liebler $k$-sets in $\PG(n,q)$ and $\AG(n,q)$. These results will be similar to the results obtained in \cite{DMSS}. 
%\begin{St}\label{H6ProjToAff}
%Every Cameron-Liebler $k$-set in PG($n,q$) that does not contain $k$-spaces at infinity, a Cameron-Liebler $k$-set of the same parameter $x$ in AG($n,q$).
%\end{St}

Here we prove Theorems \ref{H6ProjToAff} and \ref{CLkSpaceBasic_plus_lines} from the introduction.
\begin{proof}[Proof of Theorem \ref{H6ProjToAff}]
Suppose that \(\mathcal{L}\) is a Cameron-Liebler $k$-set in $\PG(n,q)$ that misses the set of $k$-spaces in \(\pi_\infty\).  So the characteristic vector of $\cL$ in $\PG(n,q)$ lies inside Im($P_n^T$)$= (\ker(P_n))^\perp$. Then, by Theorem \ref{chiUpToDown}, we obtain that for its characteristic vector $\chi_{\mathcal{L}}$ in $\AG(n,q)$ it holds that \(\chi_{\mathcal{L}} \in (\ker(A_n))^\perp=\text{ Im}(A_n^T)\). Here \(A_n\) is the point-($k$-space) incidence matrix of the affine space from Construction \ref{IncidenceMatrix}.  Due to the size of $\cL$, the parameter remains the same. This proves the assertion. %From Lemma \ref{2to8} we conclude that \(\mathcal{L}\) is a Cameron-Liebler $k$-set in $\AG(n,q)$, where the parameter is the same due to its size, and Lemma \ref{SizeAffCLkSet}.
\end{proof}
%\begin{St}\label{CLkSpaceBasic}
%	If \(\mathcal{L}\) is a Cameron-Liebler $k$-set with parameter $x$ of $\AG(n,q)$, then \(\mathcal{L}\) also is a Cameron-Liebler $k$-set of the same parameter $x$ in the projective closure $\PG(n,q)$ of $\AG(n,q)$.
%\end{St}
\begin{proof}[Proof of Theorem \ref{CLkSpaceBasic_plus_lines}]
Consider the characteristic vector $\chi_\mathcal{L}$ of the $k$-set  $\mathcal{L}$, then we know that $\chi_\mathcal{L}\in$ Im($A_n^T$). Here $A_n$ is the point-($k$-space) incidence matrix of $\AG(n,q)$. Due to Construction \ref{IncidenceMatrix}, we know that 
$$\begin{pmatrix} \chi_\mathcal{L} \\
\bar{0} \end{pmatrix} \in \text{ Im}(P_n^T),$$
with $\bar{0}$ the vector of the correct dimension that only contains zeroes. Note that $\binom{\chi_\mathcal{L}}{
\bar{0}}$ is in fact the characteristic vector of $\mathcal{L}$ in $\PG(n,q)$. So $\mathcal{L}$ is by definition a Cameron-Liebler $k$-set in $\PG(n,q)$. Due to the size of $\cL$, the parameter remains the same.
\end{proof}

\begin{St}\label{ProgToAff}
	Suppose that $\mathcal{L}$ is  a Cameron-Liebler $k$-set of parameter $x$ in $\PG(n,q)$. Then   $\mathcal{L}$ is a Cameron-Liebler $k$-set in $\AG(n,q)$ of the same parameter $x$ if and only if \(\mathcal{L}\) is skew to the set of $k$-spaces in the hyperplane at infinity $\pi_\infty$ of the affine space.
\end{St}
\begin{proof}
	Suppose that \(\mathcal{L}\) is a Cameron-Liebler $k$-set in $\PG(n,q)$ that misses the set of $k$-spaces in \(\pi_\infty\). Then, by Theorem \ref{H6ProjToAff}, we obtain that $\mathcal{L}$ is a Cameron-Liebler $k$-set in $\AG(n,q)$ that has the same parameter $x$.
	
	%The converse actually follows exactly the same ideas and arguments as Theorem 3.9 in \cite{DMSS}.
Let $\mathcal{L}$ be a Cameron-Liebler $k$-set in $\PG(n,q)$ whose restriction is a Cameron-Liebler $k$-set in $\AG(n,q)$ of the same parameter $x$. Then we can define the restriction of $\mathcal{L}$ to $\AG(n,q)$ by $\mathcal{L}'$. Using Theorem \ref{CLkSpaceBasic}, we know that $\mathcal{L'}$ is a Cameron-Liebler $k$-set in $\PG(n,q)$. So, by Lemma \ref{CLkSetProp}, it follows that $\mathcal{L}\setminus \mathcal{L}'$ is a Cameron-Liebler $k$-set of parameter $x=0$ in $\PG(n,q)$. This Cameron-Liebler $k$-set would only contain $k$-spaces in the hyperplane at infinity. So clearly $\mathcal{L}\setminus \mathcal{L}'= \emptyset$. Thus $\mathcal{L}$ does not contain $k$-spaces at infinity.
\end{proof}

\begin{St}\label{xAffineToHighProj}
	If there exists an affine Cameron-Liebler $k$-set with parameter $x$ in $\AG(n,q)$, then there exists a Cameron-Liebler $k$-set of parameter $x+\frac{q^{n-k}-1}{q^{k+1}-1}$ in the projective closure $\PG(n,q)$.
\end{St}
\begin{proof}
Due to Theorem \ref{CLkSpaceBasic}, we know that $\mathcal{L}$ is a Cameron-Liebler $k$-set in $\PG(n,q)$.
Using Lemma \ref{CLkSetProp} and Example \ref{TrivialEx}, we can extend every Cameron-Liebler $k$-set in $\AG(n,q)$ as follows: \(\mathcal{L}':= \mathcal{L}\cup \{K \in \Pi_k   \mid K \subseteq \pi_\infty\}\). This set is a Cameron-Liebler $k$-set in $\PG(n,q)$ of parameter $x+\frac{q^{n-k}-1}{q^{k+1}-1}$.
\end{proof}
\begin{Opm}
 If we now combine Theorems  \ref{ProgToAff} and \ref{xAffineToHighProj}, we find another interesting observation. Recall that every Cameron-Liebler $k$-set of parameter $x$ in $\AG(n,q)$ can be extended by adding all $k$-spaces at infinity to a Cameron-Liebler $k$-set of parameter $x+\frac{q^{n-k}-1}{q^{k+1}-1}$ in $\PG(n,q)$. But the other way also holds. Suppose we have a Cameron-Liebler $k$-set of parameter $x+\frac{q^{n-k}-1}{q^{k+1}-1}$ in $\PG(n,q)$ that contains all the $k$-spaces at infinity. Then, by Lemma \ref{CLkSetProp}, we can remove all the $k$-spaces at infinity and obtain a Cameron-Liebler $k$-set of parameter $x$ in $\AG(n,q)$.
\end{Opm}

\subsection{Equivalent definitions}
Our goal will be to prove the following theorem, which presents equivalent definitions for Cameron-Liebler $k$-sets in $\AG(n,q)$.
\begin{St}\label{EquivalencesGeneral}\label{EquivalencesAG(n,q)lines}
	Consider the affine space $\AG(n,q)$, for $n\geq 2k+1$, and let \(\mathcal{L}\) be a set of $k$-spaces such that $|\mathcal{L}|=x$\mbox{$\scriptsize{\begin{bmatrix}
	n \\
	k
	\end{bmatrix}_q}$} for a positive integer $x$. Then the following properties are equivalent.
	\begin{enumerate}
		\item \(\mathcal{L}\) is a Cameron-Liebler $k$-set in $\AG(n,q)$.
		\item For every $k$-spread $\mathcal{S}$, it holds that $|\mathcal{L}\cap \mathcal{S}|=x.$
		\item For every pair of conjugated switching $k$-sets \(\mathcal{R}\) and \(\mathcal{R}'\), \(|\mathcal{L} \cap \mathcal{R} |= |\mathcal{L} \cap \mathcal{R}'|\).
\end{enumerate}
If $k=1$ and we thus consider Cameron-Liebler line classes, then the following property is equivalent to the previous ones.
\begin{enumerate}
          \item[4.] For every line $\ell$, the number of elements of \(\mathcal{L}\) affinely disjoint to $\ell$ is equal to
	\begin{equation} \label{EqLinesEq} \left( q^2\begin{bmatrix}
	n-2\\
	1
	\end{bmatrix}_q +1 \right) (x- \chi_{\mathcal{L}}(\ell))\end{equation}
	and through every point at infinity there are exactly $x$ lines of \(\mathcal{L}\).

	\end{enumerate}
\end{St}
To do this, we will need the following statement. Note that for every subspace $\tau$ of $\AG(n,q)$, we will denote $[\tau]_k:=\{K \mid K \in \Phi_k, K\subseteq\tau\}$.

\begin{St}\label{SubspaceProp2}
Let $\mathcal{L}$ be a collection of $k$-spaces in $\AG(n,q)$ such that property (2) of Theorem \ref{EquivalencesGeneral} holds. Suppose that $\tau_A$ is an arbitrary $i$-dimensional subspace in AG($n,q$), with $i\geq \max\{k+1, 3\}$. Then property (2) also holds for $\mathcal{L}\cap [\tau_A]_k$  in the (affine) subspace $\tau_A$ with respect to the $k$-spreads in $\tau_A$.
\end{St}
\begin{proof}
Consider the projective closure PG($n,q$) of AG($n,q$) and let $\pi_\infty$ be the hyperplane at infinity.  Let $\tau_A$ be an $i$-dimensional space in $\AG(n,q)$ and let $\tau$ be its projective closure, hence $\dim(\tau\cap\pi_\infty)\geq k$. %We also fix the notation that if we consider the corresponding affine subspace of a subspace $\tau'$, we denote $\tau'_{A}$.

 Our goal is to prove that $\mathcal{L}$ restricted to $\tau_{A}$ satisfies property (2) with $x$. Pick a $(k-1)$-space $I$ in $\tau\cap \pi_\infty$ and denote $E$ as the set of affine $k$-spaces through $I$ and not in $\tau_A$. Then it is clear that every $k$-space $K\in E$ is affinely disjoint towards every $k$-space in $\tau_{A}$. It is also true that no two $k$-spaces in $E$ share an affine point, and yet as a set they cover all affine points not in $\tau_A$. So if we would choose a $k$-spread $\mathcal{S}$ in $\tau_{A}$, then we can always extend this $k$-spread to a $k$-spread in AG($n,q$) in the following way
$$\mathcal{S}':=\mathcal{S}\cup E.$$
Note that for every $k$-spread $\mathcal{S}$ in $\tau_A$, we can use the same $E$. So we have that 
$$x=|\mathcal{L}\cap\mathcal{S}'|=|\mathcal{L}\cap \mathcal{S}|+|\mathcal{L}\cap E|.$$
%Hence if we call $\mathcal{L}_{\pi_{A}}$ the restriction of $\mathcal{L}$ to $\pi_{A}$, we have for every $\mathcal{S}$ that $|\mathcal{L}_{\pi_{A}}\cap \mathcal{S}|=x-|\mathcal{L}\cap E|$. But this last factor is in fact a constant, due to the fact that $E$ was fixed for every $\mathcal{S}$. This proves the theorem.
Hence, we have for every $k$-spread $\mathcal{S}$ in $\tau_A$ that $|(\mathcal{L}\cap[{\tau_{A}}]_k)\cap \mathcal{S}|=x-|\mathcal{L}\cap E|$. But this last term is a constant, due to the fact that $E$ was fixed for every $k$-spread $\mathcal{S}$. This proves the theorem.
\end{proof}

 We now can prove the main theorem.
\begin{proof}[Proof of Theorem \ref{EquivalencesGeneral}]
We first prove the equivalence between the first 3 statements and then we prove the equivalence with statement 4.
\begin{itemize}
\item {\bf From (1) to (3):} Suppose that $\cL$ is a Cameron-Liebler $k$-set  with characteristic vector $\chi_\cL$. Let $\mathcal{R}$ and $\mathcal{R}'$ be a pair of conjugated switching sets with characteristic vectors $\chi_\mathcal{R}$ and $\chi_{\mathcal{R}'}$ respectively. Then it holds, due to the definition of a pair of conjugated switching sets, that $$\chi_{\mathcal{R}}-\chi_{\mathcal{R}'}\in \ker(A_n),$$
where $A_n$ is the point-line incidence matrix of $\AG(n,q)$. Since $\chi_\cL\in \text{Im}(A_n^T)=(\ker(A_n))^\perp$, we have that $$\chi_\cL\cdot (\chi_{\mathcal{R}}-\chi_{\mathcal{R}'})=0.$$
This concludes the statement.

\item  {\bf From (3) to (2):} Since for every pair of $k$-spreads $\mathcal{S}_1$ and $\mathcal{S}_2$, it holds that $\mathcal{S}_1\setminus \mathcal{S}_2$ and $\mathcal{S}_2\setminus \mathcal{S}_1$ are a pair of conjugated switching sets, we know that $|\mathcal{L}\cap \mathcal{S}_1|=c=|\mathcal{L}\cap\mathcal{S}_2|$. So we only need to show that $c=x$. To obtain this we double count the pairs $(K, \mathcal{S})$, with $K\in \mathcal{S}\cap \cL$ and $\mathcal{S}$ a $k$-spread of type II. Hence, due to the fact that $|\cL|=x\gauss{n}{k}$ and the number of $k$-space of type II through a fixed $k$-spread equals $1$, we get that
$$x\gauss{n}{k}\cdot 1= \gauss{n}{k}\cdot c.$$
Thus $x=c$, which completes the assertion.

\item {\bf From (2) to (1):} If $n=2k+1$, then we know, due to Lemma \ref{SpreadsAG(n,q)} (1) and Theorem \ref{EquivalenceProj}, that $\cL$ is a Cameron-Liebler $k$-set of parameter $x$ in $\PG(n,q)$. Hence due to Theorem \ref{H6ProjToAff}, the assertion follows.\\
Suppose now that $n>2k+1$. 
Here we will use similar techniques as in the proof of Theorem 2.9 in \cite{DBD2018}, (more specifically in the step (7) to (1)). Let $\tau$ be an arbitrary $(2k+1)$-dimensional subspace, then we can consider $\cL\cap [\tau]_k$ in the affine space $\tau$. Due to Theorem \ref{SubspaceProp2}, we obtain that $\cL\cap [\tau]_k$ satisfies Property (2) in $\tau$. Hence, using the previous observation, we obtain that $\cL\cap [\tau]_k$ is a Cameron-Liebler $k$-set of a certain parameter in $\tau$. Note that this space $\tau$ was chosen arbitrarily. Thus it follows for every $(2k+1)$-dimensional subspace of $\AG(n,q)$ that, for the characteristic vector of $\cL\cap [\tau]_k$, it holds that 
$$\chi_{\cL\cap [\tau]_k}\in \text{Im}(A_\tau^T),$$
with $A_\tau$ the incidence matrix of $\tau$. So we have that $\chi_{\cL\cap [\tau]_k}$ is a linear combination of the rows of $A_\tau$. Note that due to the fact that $A_\tau$ has full row rank, it holds that this linear combination is unique. We only need to show that $\chi_\cL$ is uniquely defined by the vectors $\chi_{\cL\cap [\tau]_k}$, with $\tau$ varying over all $(2k+1)$-spaces in $\AG(n,q)$. \\
We first want to show that for every two $(2k+1)$-spaces $\tau$ and $\tau'$ the coefficients of the row corresponding to a point in $\tau\cap \tau'$ in the linear combination of $\chi_{\cL\cap [\tau]_k}$ and $\chi_{\cL\cap [\tau']_k}$ are equal.\\
Consider the subspace $\tau\cap\tau'$, and consider the corresponding columns of $A_n$. Then using the fact that $A_{\tau\cap \tau'}$ also has full row rank, we conclude that the linear combination of the rows that give $\chi_{\cL\cap [\tau\cap \tau']_k}$ is unique. Note that this unique linear combination has the same coefficients for the rows corresponding
with points in $\tau\cap \tau'$ as $\chi_{\cL\cap [\tau]_k}$ and $\chi_{\cL\cap [\tau']_k}$ has respectively. Here we also used the fact that an entry of $A_n$ corresponding with a point of $\tau \setminus \tau'$ or $\tau'\setminus\tau$ and a $k$-space in $\tau\cap \tau'$ is zero. Thus we may conclude that the common rows in $\chi_{\cL\cap [\tau]_k}$ and $\chi_{\cL\cap [\tau']_k}$ have the same coefficient.

Using all of these $(2k+1)$-spaces, we have that $\chi_\cL$ is uniquely defined and $\chi_\cL\in \text{Im}(A_n^T)$. This proves the assertion.

\item {\bf Equivalence between (4) and the rest, when $k=1$:}
First if \(\mathcal{L}\) is an affine Cameron-Liebler line class with parameter $x$, then, by Theorem \ref{CLkSpaceBasic}, we get that \(\mathcal{L}\) is a Cameron-Liebler line class in $\PG(n,q)$. Here we know that for every (affine) line $\ell$, there are exactly
	$$ q^2 \begin{bmatrix}
	n-2\\
	1
	\end{bmatrix}_q (x-\chi_{\mathcal{L}}(\ell))$$
	lines of \(\mathcal{L}\) projectively disjoint to $\ell$. So we only still need to consider the lines of $\mathcal{L}$ through the point $\ell\cap \pi_\infty$. But since this is a point at infinity, which gives a line spread of type II, we have a total of $x$ lines of $\mathcal{L}$ through this point. Thus if we add those \(x-\chi_{\mathcal{L}}(\ell)\) lines of \(\mathcal{L}\) not equal to $\ell$, we get a total of
	$$ \left( q^2 \begin{bmatrix}
	n-2\\
	1
	\end{bmatrix}_q +1 \right) (x-\chi_{\mathcal{L}}(\ell))$$
	lines of \(\mathcal{L}\) disjoint to $\ell$ in AG($n,q$).
	
	 Conversely, suppose that Property (2) holds, then we look at the corresponding projective space $\PG(n,q)$.  We can see that of the 
	$$ \left( q^2 \begin{bmatrix}
	n-2\\
	1
	\end{bmatrix}_q +1 \right) (x-\chi_{\mathcal{L}}(\ell))$$
	lines of \(\mathcal{L}\) that are disjoint in $\AG(n,q)$ to an affine line $\ell$, there are
	$$ q^2 \begin{bmatrix}
	n-2\\
	1
	\end{bmatrix}_q (x-\chi_{\mathcal{L}}(\ell))$$
	elements of \(\mathcal{L}\) projectively disjoint to $\ell$. 

If we now pick a line $\ell$ in \(\pi_\infty\), then there are
	 $$\frac{q^n-1}{q-1}-(q+1)= q^2 \begin{bmatrix}
	 n-2\\
	 1
	 \end{bmatrix}_q$$
	 points in \(\pi_\infty\) not in $\ell$. Through every such point, there are exactly $x$ lines of \(\mathcal{L}\) that are disjoint to $\ell$. If we combine these results we obtain that, by Theorem \ref{EquivalenceProj}, it follows that \(\mathcal{L}\) is a  Cameron-Liebler line class in $\PG(n,q)$ with parameter $x$. Using  Theorem \ref{ProgToAff}, we see that \(\mathcal{L}\) is also a Cameron-Liebler line class in $\AG(n,q)$ with the same parameter $x$.	

%\item  From (3) to (1): Consider the corresponding projective space $\PG(n,q)$. If we now use Lemma \ref{CongSetsUp}, then we see that every pair of conjugated switching sets of $\PG(n,q)$ defines a pair of conjugated sets in $\AG(n,q)$. But since $\mathcal{L}$ is a set of affine $k$-spaces, we know that every pair of conjugated switching sets in $\PG(n,q)$ will have the same intersection size with $\mathcal{L}$. Thus \(\mathcal{L}\) defines a Cameron-Liebler $k$-set in $\PG(n,q)$. By Theorem \ref{EquivalenceProj}, we know that the characteristic vector of $\mathcal{L}$ in PG($n,q$)  equal to $\widetilde{\chi}_{\mathcal{L}}$  lies in $(\ker(P_n))^\perp$. This space is the dual of the kernel of the point-($k$-space) incidence matrix of $\PG(n,q)$. Due to the fact that $\mathcal{L}$ has no $k$-spaces at infinity, we may use Theorem \ref{chiUpToDown} to obtain that the characteristic vector $\chi_\mathcal{L}$ in $\AG(n,q)$ satisfies property (1).

%From Lemma \ref{ConstInt} we know that (1) implies (2).
\end{itemize}

\end{proof}
\begin{Opm}
There is also another way to prove the equivalence of (1) and (2) for  $k=1$. For this, we will use association schemes. This will be done in Section 5.
\end{Opm}

\begin{Le}\label{CongSetsUp}
	%If \(\mathcal{R}\) and \(\mathcal{R}'\) are conjugated switching $k$-sets in $\PG(n,q)$, then their restrictions to $\AG(n,q)$ also form conjugated switching $k$-sets in $\AG(n,q)$.
Suppose that $\mathcal{R}$ and $\mathcal{R}'$ are a pair of conjugated switching $k$-sets in $\PG(n,q)$. If we define $\mathcal{R}_A$ (and $\mathcal{R}'_A$) as the set of affine $k$-spaces of $\mathcal{R}$ (and $\mathcal{R}'$ respectively), then $\mathcal{R}_A$ and $\mathcal{R}_A'$ are conjugated switching $k$-sets in $\AG(n,q)$.
\end{Le}
\begin{proof}
	 Since $ \mathcal{R} \cap \mathcal{R}'=\emptyset$, it is clear that
	$$\mathcal{R}_A\cap \mathcal{R}_A' = \emptyset.$$
	Since \(\mathcal{R}_A\)  and \(\mathcal{R}_A'\) arose from $\mathcal{R}$ and $\mathcal{R'}$, we know that no two $k$-spaces in the same set intersect. Thus both are still partial $k$-spreads. So we only need to show that they still cover the same set of points. If an affine point $p$ is covered by \(\mathcal{R}_A \), then this point (which is also a projective point) is also covered by $\mathcal{R}$ and, hence, by \(\mathcal{R'}\). Since this point was affine, the corresponding $k$-space of $\mathcal{R'}$ is contained in $\mathcal{R}'_A$ and, hence, the point is covered by $\mathcal{R}'_A$.
\end{proof}

This lemma also shows that there exist conjugated switching sets in $\AG(n,q)$, since they exist in $\PG(n,q)$. This fact implies that Theorem \ref{EquivalencesGeneral} does not have a trivial assumption.

\section{The association scheme of affine lines}\label{sec:assoc}
Our goal in this section is that we want to investigate the association scheme of lines in $\AG(n,q)$. We start with repeating some definitions of association schemes. If the reader is not familiar with association schemes, we refer to \cite{DistanceRegGraph,Godsil}.
 \begin{Def}{\cite[Section 2.1]{DistanceRegGraph}}
	Let $X$ be a finite set. A $d$\emph{-class   association scheme} is a pair $(X, \mathcal{R})$, where $\mathcal{R}=\{\mathcal{R}_0,\mathcal{R}_1,..., \mathcal{R}_d\}$ is a set of binary symmetrical relations with the following properties:
	\begin{enumerate}
		\item $\{\mathcal{R}_0,\mathcal{R}_1,..., \mathcal{R}_d\}$ is a partition of $X\times X$.
		\item \(\mathcal{R}_0\) is the identity relation.
		\item There exist constants \(p_{ij}^l\) such that for \(x,y\in X\), with \((x,y)\in \mathcal{R}_l\), there are exactly \(p_{ij}^l\) elements \(z\) with \((x,z)\in \mathcal{R}_i\) and \((z,y)\in \mathcal{R}_j\). These constants are called the \emph{intersection numbers} of the association scheme.
		%\item For each $k,i$ and $j$ there holds that
		%\[p^k_{ij}=p^k_{ji}.\]
	%	\item for each $k,i$ and $j$ there holds that
	%	\[p^k_{ij}=p^k_{ji}.\] 
	\end{enumerate}
\end{Def}
In such a $d$-class association scheme we can define adjacency matrices as follows.
 \begin{Def}
Consider a $d$-class association scheme $(X,\mathcal{R})$, where $\mathcal{R}=\{\mathcal{R}_0,\mathcal{R}_1,..., \mathcal{R}_d\}$ and $X=\{x_1,...,x_n\}$. Then we can define $d+1$ matrices $B_0,...,B_d$ of dimension $n \times n$, such that
$$(B_k)_{ij}= \left\{ \begin{aligned} 1, & \text{ if } (x_i,x_j)\in \mathcal{R}_k \\ 0, & \text{ if } (x_i,x_j)\not\in \mathcal{R}_k. \end{aligned} \right.$$
These matrices are called the \emph{adjacency matrices} of the association scheme.
\end{Def}
An important property of these adjacency matrices is that they can be diagonalized simultaneously, so we obtain maximal common (right) eigenspaces $V_0,...,V_d$.
 It is also known that these adjacency matrices span a $(d+1)$-dimensional commutative $\mathbb{C}$-algebra $\mathcal{A}$. This algebra is called the \emph{Bose-Mesner algebra}, which has a basis of idempotents $\{E_i \mid 0 \leq i \leq d\}$. One can prove that every matrix $E_i$ is the orthogonal projection to the eigenspace $V_i$. If we would consider the common eigenspaces, we can denote all the eigenvalues in a matrix. This matrix is called the eigenvalue matrix.

\begin{Def}
Consider a $d$-class association scheme  $(X,\mathcal{R})$, where $\mathcal{R}=\{\mathcal{R}_0,\mathcal{R}_1,..., \mathcal{R}_d\}$ and $X=\{x_1,...,x_n\}$. Let $B_0,...,B_d$ be the adjacency matrices and $\{E_i \mid 0\leq i\leq d\}$ be the idempotent basis of the Bose-Mesner algebra. Then the \emph{eigenvalue matrix} $P=[P_{ij}]$ and the \emph{dual eigenvalue matrix} $Q=[Q_{ij}]$ are the matrices for which it holds that
$$ B_j= \sum_{i=0}^d P_{ij}E_i \text{   and  } E_j= \frac{1}{n} \sum_{i=0}^d Q_{ij}B_i.$$
Here $0 \leq i,j \leq d$.
\end{Def}
Since every $E_i$ in the idempotent basis gives an orthogonal projection onto $V_i$, it is indeed true that the values $P_{ij}$ are the eigenvalues. Another important fact is that $PQ= n I_{d+1}=QP$. 

We now give a well-known example of such an association scheme.

\begin{Ex}\cite[Example 1.1.2]{Godsil}
Consider the set of lines in $\PG(n,q)$, with $n\geq 3$. Then this is a finite set, which we will call $\Pi_1$. Consider now the following set of relations $\mathcal{R}'=\{\mathcal{R}_0', \mathcal{R}_1',\mathcal{R}_2'\}$. Then for $\ell$ and $\ell'$ in $\Pi_1$, we have that
\begin{itemize}
\item  \((\ell,\ell')\in \mathcal{R}_0'\) if \(\ell=\ell'\).
	\item \((\ell,\ell')\in\mathcal{R}_1'\) if they meet in a point.
	\item \((\ell,\ell')\in\mathcal{R}_2'\) when they do not meet at all.
\end{itemize}
It is well-known that $\Delta'=(\Pi_1, \mathcal{R}')$ gives an association scheme. This concept can be generalized to $k$-spaces in $\PG(n,q)$.
\end{Ex}
%It is well known that in general the sets of $k$-spaces in PG($n,q$), where the relations are defined by the intersection dimensions, give an association scheme in PG($n,q$). But if we consider AG($n,q$) this totally changes, since the intersection of two $k$-spaces in AG($n,q$) can lie at infinity or not. This will imply an increase of relations. We will give this type of association scheme for lines. 
We try to define a similar association scheme for lines in $\AG(n,q)$. Note that due to the fact that there exists a concept of infinity in $\AG(n,q)$, this will lead to an increase of relations. Here we see that relation $\mathcal{R}_1'$ will split into two separate relations.
\begin{Cons}\label{AssoScheme}
Consider the set $\Phi_1$ of lines of $\AG(n,q)$, with $n\geq 3$. Then we can define a \emph{3-class association scheme} 
$\Delta=(\Phi_1, \mathcal{R})$, where we denote the following relations \(\mathcal{R}=\{\mathcal{R}_0,\mathcal{R}_1,\mathcal{R}_2,\mathcal{R}_3\}\) as follows. Pick \(\ell,\ell' \in \Phi_1\), then
\begin{itemize}
	\item  \((\ell,\ell')\in \mathcal{R}_0\) if \(\ell=\ell'\).
	\item \((\ell,\ell')\in\mathcal{R}_1\) if they meet in an affine point.
	\item \((\ell,\ell')\in\mathcal{R}_2\) if they meet in a point at infinity.
	\item \((\ell,\ell')\in\mathcal{R}_3\) when they do not meet in the corresponding projective space.
\end{itemize}
In order to prove that this is an association scheme, we can refer to \cite[Chapter 4]{JonathanThesis}, where the intersection numbers were explicitly calculated. Another way to view this, is as a semilattice and conclude, due to \cite{Delsarte1976}, that $\Delta$ is indeed an association scheme.
\end{Cons}
%\begin{Le}
%	The pair \(\Delta=(\Phi_1,\mathcal{R})\) from Construction \ref{AssoScheme} is a 3-class association scheme.
%\end{Le}
%\begin{proof}
%This can be done by checking the axioms, where axiom 3 will be the hardest. But since we will need all the intersection numbers of the association scheme, we have calculated them in \cite{JonathanThesis}. Another way to view this is, is as a semi lattice and conclude due to \cite{Delsarte1976} that $\Delta$ is indeed an association scheme.
%\end{proof}
Let us consider $\Delta$. If we number the lines of $\AG(n,q)$ in a fixed order $$\left\lbrace \ell_i \mid i \in \left\lbrace 0,..., \frac{q^{n-1}(q^n-1)}{(q-1)}-1\right\rbrace \right\rbrace, $$
then we can define the adjacency matrices as \(B_0,B_1,B_2\) and \(B_3\). We know that these are $\frac{q^{n-1}(q^n-1)}{(q-1)} \times \frac{q^{n-1}(q^n-1)}{(q-1)}$ matrices  over \(\mathbb{C}\) that have common (right) eigenspaces.
If we define these common (right) eigenspaces  by  \(V_0,V_1,V_2\) and \(V_3\), then we know that \(\mathbb{C}^{\Phi_1}= V_0 \perp V_1 \perp V_2 \perp V_3 \). Consider now the Bose-Mesner algebra $\mathcal{A}$ of the association scheme $\Delta$, which will be a $4$-dimensional $\mathbb{C}$-algebra. Then we know that $\mathcal{A}$ has  a basis of idempotents $\{E_i \mid 0\leq i\leq 3\}$, such that every $E_i$ is the orthogonal projection onto $V_i$.

\subsection{Calculating the eigenvalue matrix and dual eigenvalue matrix of $\Delta$}
In order to find the eigenvalue matrix $P$ and the dual eigenvalue matrix $Q$, we need to define some other matrices known as the intersection matrices.
\begin{Def}
	Consider a $d$-class association scheme with intersection numbers $p_{ij}^k$. Then we can define the following $(d+1)\times (d+1)$ matrices for \(i \in \{0,...,d\}\)
	$$\mathcal{P}_i=[p_{ij}^k]_{k,j},$$
	hence the $(k,j)$-entry is $\mathcal{P}_i(k,j)=p_{ij}^k$. These matrices are known as \emph{intersection matrices}. 
\end{Def}
These intersection matrices for the association scheme of Construction \ref{AssoScheme} can be calculated:
\[ \mathcal{P}_0 = \begin{pmatrix}
	1 & 0 & 0 & 0 \\
	0 & 1 & 0 & 0\\
	0 & 0 & 1 & 0 \\
	0 & 0 & 0 & 1
	\end{pmatrix},\]
\[\mathcal{P}_1= \begin{pmatrix}
	0 & q\left( \frac{q^n-1}{q-1}-1\right)  & 0 & 0 \\
	
	1 & (q-1)^2+\left( \frac{q^n-1}{q-1}-2 \right) & q-1 & (q-1)\left( \frac{q^n-1}{q-1}-1-q \right) \\
	
	0 & q^2 & 0 & q\left( \frac{q^n-1}{q-1}-1-q \right)\\
	
	0 & q^2 & q & q\left( \frac{q^n-1}{q-1}-1-(q+1) \right)\\
	\end{pmatrix},\]
\[\mathcal{P}_2= \begin{pmatrix}
	0 & 0 & q^{n-1}-1 & 0 \\
	0 & q-1 & 0 & q^{n-1}-q\\
	1 & 0 & q^{n-1}-2 & 0 \\
	0 & q & 0 & q^{n-1}-1-q\\
	\end{pmatrix}\]
and
\[\mathcal{P}_3=\left(\begin{array}{rrrr}
	0 & 0 & 0 & \frac{q^{2} - {\left(q + 1\right)} q^{n} +
		q^{2  n-1}}{q - 1} \\
	0 & -q^{2} + q^{n} & -q + q^{n-1} & \frac{q^{3}
		+ q^{2} - {\left(2 \, q^{2} + q - 1\right)} q^{n-1} - q + q^{2 
			n-1}}{q - 1} \\
	0 & -\frac{q^{3} - q^{n + 1}}{q - 1} & 0 & \frac{q^{3} +
		q^{2} - {\left(2 \, q + 1\right)} q^{n} + q^{2 n-1}}{q - 1} \\

	1 & -\frac{q^{3} + q^{2} - q - q^{n + 1}}{q - 1} & -q
		- 1 + q^{n-1} & \frac{q^{3} + 3 \, q^{2} - {\left(2 \, q^{2} + 2
			\, q - 1\right)} q^{n-1} - 2 \, q+ q^{2 n-1}}{q - 1}
	\end{array}\right).\]
For these calculations we refer to \cite[Chapter 4]{JonathanThesis}.
\begin{Le}\cite[page 45, Lemma 2.2.1]{DistanceRegGraph}\label{ConstrOfP}
	Consider a $d$-class association scheme together with the eigenvalue matrix $P$ and the intersection matrices $\mathcal{P}_i$, for $i \in \{0,...,d\}$. Then
	$$P \cdot \mathcal{P}_i \cdot P^{-1}= \begin{pmatrix}
	P_{0i} & 0 & 0 & ... & 0 \\
	0 & P_{1i} & 0 & ... & 0 \\
	\vdots & \vdots & \vdots & \ddots & \vdots \\
	0 & 0&0 &... & P_{di}
	\end{pmatrix}.$$
	Consequently, $\mathcal{P}_i$ and the adjacency matrix $B_i$ have the same eigenvalues.
	%Consider  a $d$-class association scheme together with its adjacency matrices \(A_0,A_1,...,A_d\) and the intersection matrices \(B_0,B_1,...,B_d\). Then the minimal polynomials and thus the eigenvalues of each matrix \(A_i\) and corresponding matrix \(B_i\) are the same.
\end{Le}
This lemma implies that the intersection matrices can be diagonalized simultaneously.
In order to find $P$, we use the following theorem.
%\begin{St}{(\cite[Proposition 2.2.2]{DistanceRegGraph})}
	%Consider a $d$-class association scheme with intersection numbers $p_{ij}^k$.
	%Let $v_i$ and \(u_i\) be the common (normalized) right and left eigenvectors of the intersection matrices, then

	%$$u_i^T= \left( \begin{pmatrix}
	%p_{00}^0 & 0 & ... & 0 \\
	%0 & p_{11}^0 & ... & 0 \\
	%\vdots & \vdots & \ddots & \vdots \\
	%0 & 0 & 0 & p_{dd}^0 \\
	%\end{pmatrix}\cdot v_i\right) ^T. $$
	% Also we conclude that the rows of the eigenvalue matrix $P$ are the elements $(u_i)^T$.
%\end{St}
\begin{St}{\cite[Proposition 2.2.2]{DistanceRegGraph}}
Consider a $d$-class association scheme and let $u_i$, for $i \in \{0,...,d\}$, be the set of common left normalized (column) eigenvectors of the intersection matrices. Here we mean with normalized, that $(u_i)_0=1$ for every $i$. Then  the rows of the eigenvalue matrix $P$ are the elements $(u_i)^T$.
\end{St}
This lemma together with the following left (normalized) eigenvectors of the intersection matrices above, will give us the eigenvalue matrix.

\begin{equation*}
	u_0= \left( 1 , -q + \frac{q^{n + 1} - 1}{q - 1} - 1 , q^{n - 1} - 1 , q
	+ q^{n} + \frac{q^{2 n - 1} - 1}{q - 1} - \frac{2 \, {\left(q^{n + 1}
			- 1\right)}}{q - 1} + 1\right)^T, 
\end{equation*}
\begin{equation*}
	u_1=\left( 1 ,-q + \frac{q^{n} - 1}{q - 1} - 1 , -1 , q - \frac{q^{n}
		- 1}{q - 1} + 1\right)^T, 
\end{equation*}
\begin{equation*}
 u_2=\left( 1,-q,-1,q\right)^T ,
\end{equation*}
\begin{equation*}
	u_3=\left(  1 , -q , q^{n - 1} - 1 , q - q^{n - 1}\right)^T. 
\end{equation*}

These left eigenvectors were calculated by using Sage \cite{sagemath}. From this together with the lemma above, we can obtain the eigenvalue matrix $P$ of the association scheme $\Delta$, see Construction \ref{AssoScheme}. So we conclude that
\begin{equation}\label{MatrixP}P=   
\left(\begin{array}{rrrr}
1 & -\frac{q^{2} - q^{n + 1}}{q - 1} & q^{n-1} - 1&
\frac{q^{2} - {\left(q + 1\right)} q^{n} + q^{2  n-1}}{q - 1} \\
1 & -\frac{q^{2} - q^{n}}{q - 1} & -1 & \frac{q^{2} -
	q^{n}}{q - 1} \\
1 & -q & -1 & q \\
1 & -q & q^{n-1} - 1 & q - q^{n-1} \\
\end{array}\right)\end{equation}
and due to \(PQ=q^{n-1}\left( \frac{q^n-1}{q-1}\right) I_4= |\Phi_1| I_4=QP\), we obtain that
\begin{equation}\label{MatrixQ}
Q=\left(\begin{array}{rrrr}
1 & q^{n} - 1 & -\frac{{\left(q^{2} + 1\right)} q^{n} - q^{2} - q^{2 \,
		n}}{q^{2} - q} & \frac{q^{n} - q}{q - 1}  \\
1 & \frac{{\left(q^{2} + 1\right)}
	q^{n} - q^{2} - q^{2 \, n}}{q^{2} - q^{n + 1}} & -\frac{{\left(q^{2} + 1\right)} q^{n} - q^{2} - q^{2 \,
		n}}{q^{2} - q^{n + 1}} & -1  \\
1 & \frac{q - q^{n + 1}}{q^{n} - q} & \frac{{\left(q^{2} + 1\right)} q^{n} - q^{2} - q^{2 \,
		n}}{{\left(q - 1\right)} q^{n} - q^{2} + q} & \frac{q^{n} - q}{q -
	1}  \\
1 & \frac{q - q^{n +
		1}}{q^{n} - q} &\frac{q^{n + 1} - q}{q^{n} - q} & -1 
\end{array}\right).\end{equation}

\section{Cameron-Liebler line classes in $\AG(n,q)$}
%Recall that many results for Cameron-Liebler $k$-spaces in $\AG(n,q)$ depend on the fact that $(k+1)\mid (n+1)$. In this section we try to remove this claim for Cameron-Liebler line classes in $\AG(n,q)$.
%Our first goal in this section is to prove the following theorem, which, due to Theorem \ref{EqAffk+1|n+1}, is already valid for $2 \mid (n+1)$.
In this section we will give an alternative proof for the following statement, which is a special case of Theorem \ref{EquivalencesGeneral}.
\begin{St}\label{MAIN}
%Consider a Cameron-Liebler line class in $\AG(n,q)$, $n\geq 3$, with characteristic vector $\chi_\cL$. Then $\chi_\cL \in (\ker(A_n))^\perp$, where $A_n$ is the point-line incidence matrix of $\AG(n,q)$.
Suppose that $\mathcal{L}$ is a set of lines in $\AG(n,q)$, $n\geq 3$, such that for every line spread $\mathcal{S}$ it holds that 
$$|\mathcal{L}\cap \mathcal{S}|=x.$$
Then $\mathcal{L}$ is a Cameron-Liebler line class of parameter $x$.
\end{St}
To prove this theorem, we make use of the association scheme of Section 4. Let us recall $\Delta$ from Construction \ref{AssoScheme}, then we first start with the concept of inner distributions.

\subsection{Inner distribution}
\begin{Def}(\cite[Section 2.5]{DistanceRegGraph} and \cite[Definition: Section 5, (10)]{MT2009})
	Consider a $d$-class association scheme $(X, \mathcal{R})$ and let $\cL$ be a subset of $X$, then we can consider its characteristic vector $\chi_\cL$. For this vector we can define its \emph{inner distribution} as the row vector \(u=\left( u_0, \, u_1, \, u_2, \, ..., \, u_d\right)\) with elements in $\mathbb{R}$, for which it holds that
	$$u_i= \frac{1}{|\cL|}|\mathcal{R}_i \cap (\cL \times \cL)|.$$ 
	%Consider a set of lines \(L\), then we consider its characteristic vector \(\chi_L\). For this vector we define its \emph{inner distribution} as the vector \(u=\left( u_0, \, u_1, \, u_2, \, u_3\right) \), for which it holds that
	%$$u_i= \frac{1}{|L|}|\mathcal{R}_i \cap (L \times L).|$$
\end{Def}
\begin{Opm}
	Note that for the inner distribution \(u=\left( u_0, \, u_1, \, u_2, \, ..., \, u_d\right) \) of a certain characteristic vector $\chi_\cL$, it holds that
	$$u_i= \frac{1}{|\cL|} \chi_\cL^T \cdot B_i \cdot \chi_\cL,$$
	for $0 \leq i\leq d$.% Here $A_i$ denotes the incidence matrix of relation $i$ in the association scheme.
\end{Opm}
The following theorem will give us a way to observe in which eigenspaces of $\Delta$ a characteristic vector lies in.
\begin{St}(\cite[Lemma 2.5.1 and Proposition 2.5.2]{DistanceRegGraph})\label{InnerDSpan}
	Consider a $d$-class association scheme $\Gamma=(X, \mathcal{R})$ and let $\mathcal{A}$ be its Bose-Mesner algebra. Denote the idempotent basis of $\mathcal{A}$ by $\{E_i \mid 0\leq i\leq d\}$, with common eigenspaces $V_0,...,V_d$. Then it follows for every subset $\cL$ of $X$, that its characteristic vector $\chi_\cL \in \mathbb{R}^d$ can be written as follows
		$$\chi_\cL= a_0 v_0+a_1v_1+\cdots+a_dv_d,$$
	with \(v_i \in V_i\) and \(a_i \in \mathbb{R}\) for each \(0\leq i \leq d\). If $u$ is the inner distribution of $\chi_\cL$, then the following properties are equivalent for fixed $0\leq i \leq d$
\begin{enumerate}
\item \((u \cdot Q)_i=0\), with $Q$ the dual eigenvalue matrix of $\Gamma$.
\item$E_i\cdot \chi_\cL=0$.
\end{enumerate}
This last property implies that the projection of $\chi_\cL$ onto the eigenspace $V_i$ is zero, thus $a_i=0$.
\end{St}
Now we mention the next very useful theorem stated in \cite{Delsarte1976}. Our formulation is based on unpublished notes by Klaus Metsch.

\begin{St}\cite[Theorem 6.8]{Delsarte1976}\label{Conclusiontheorem0}
	Let $\Gamma=(X, \mathcal{R})$ be a $d$-class association scheme, with $\{E_i \mid 0\leq i\leq d\}$ the idempotent basis of the Bose-Mesner algebra. Suppose $G$ is a subgroup of $Aut(\Gamma)$ that acts transitively on $X$ and whose orbits on $X\times X$ are the relations $\mathcal{R}_0,..., \mathcal{R}_d$.  Let $\chi$ and $\psi$ be vectors of $\mathbb{R}^{|X|}$. Then the following two statements are equivalent.
	\begin{enumerate}
		\item For all $k\geq 1$, we have $E_k\cdot \chi=0$ or $E_k\cdot \psi=0$.
		\item  \( \chi\cdot \psi^g\) is constant for all \(g \in G\).
	\end{enumerate}
	
\end{St}
\begin{Opm}
	We know that property (1) is equivalent with the fact that both vectors lie in opposite (common) eigenspaces besides $V_0$. 	
\end{Opm}
\begin{Opm}
	A second observation is that in the $3$-class association scheme $\Delta$, the group $\AGL(n,q)$ acts indeed transitively on pairs of lines of the same type in $\AG(n,q)$. It is also clear that elements of $\AGL(n,q)$ send line spreads of type I and type II to line spreads of type I and type II respectively.
\end{Opm}
The same happens for line spreads of type III, we explicitly proved this fact. But first we give a definition.
\begin{Def}
Let $\mathcal{S}$ be a line spread of type III, with the property that all the chosen points $p_i$ in $\pi_{n-2}$ are chosen differently. Then we call $\mathcal{S}$ a line spread of type $\text{III}^+$.
\end{Def}
\begin{Le}\label{ThetaMaintainTypeIII}
  	The affine collineation group AGL($n,q$) sends spreads of type III to spreads of type III. In particular, it sends spreads of type $\text{III}^+$ to spreads of type $\text{III}^+$.
 \end{Le}
\begin{proof}
	Consider \(\mathcal{S}\) to be a line spread of type III, defined by an $(n-2)$-space $\pi_{n-2}\subseteq \pi_\infty$, the set of hyperplanes $H= \{\pi_i \mid i \in \{1,...,q \} \}$ and the $q$ points \(p_i\in \pi_{n-2}\). If we now consider $\theta \in$ AGL($n,q$), then \(\pi_{n-2}^\theta\subseteq \pi_\infty\) and all the hyperplanes of \(H= \{ \pi_i \mid i \in \{1,...,q\}\}\) are sent to different hyperplanes through \(\pi_{n-2}^\theta\). Also all the points \(p_i\) are sent to points \(p_i^\theta\in \pi_{n-2}^\theta\), which if they all are different points they shall remain so. We conclude that 
	$$\mathcal{S}^\theta= \{K^\theta \in \Phi_1 \mid p_i\in K \subseteq \pi_i \text{ for some } i \}=\{K' \in \Phi_1 \mid p_i^\theta\in K' \subseteq \pi_i^\theta \text{ for some } i \},$$
	which is of the required form.
\end{proof}

\subsection{About the common eigenspaces}

In this section we give a basis for $V_0\perp V_1$ and $V_0\perp V_3$, and give a spanning set for $V_0\perp V_2 \perp V_3$ in the association scheme $\Delta$ from Construction  \ref{AssoScheme}.
\begin{Def}
A point-pencil in $\PG(n,q)$ or $\AG(n,q)$ is the set of lines through a fixed point in $\PG(n,q)$ or $\AG(n,q)$ respectively.
\end{Def}
\begin{St}(\cite[Theorem 9.5]{DeBruyn})\label{FullRankIncidence}
	The point-line incidence matrix of $\AG(n,q)$  and $\PG(n,q)$ has full rank, which equals the number of points in $\AG(n,q)$ and $\PG(n,q)$ respectively. Hence the rows of these incidence matrices, which correspond to points and give point-pencils are linearly independent.
\end{St}
\begin{Le}\cite[Lemma 2.2.1 (ii)]{DistanceRegGraph}\label{Span}
		Consider the dual eigenvalue matrix $Q$ in an association scheme, then \(Q_{0i}= \dim(V_i)\). %Here $V_i$ is the $i$-th common eigenspace of the association scheme.
\end{Le}
We now prove the following theorem that characterizes the space $V_0\perp V_1$.
\begin{St}\label{pnt-pnc}
 	Consider the affine space AG($n,q$) and the $3$-class association scheme $\Delta$ (see Construction \ref{AssoScheme}). Then the point-pencils form a basis of the space $V_0\perp V_1$.
 \end{St}
\begin{proof}
	Let us first find the inner distribution of a point-pencil. It can be seen that this is equal to
 	\[ 
 	u=\left(1,\,\frac{q^{n} - q}{q - 1},\,0,\,0\right).\]
Thus we obtain that \[u\cdot Q=\left(\frac{q^{n} - 1}{q - 1},\, -\frac{{\left(q + 1\right)}
	q^{n-1} - 1 - q^{2  n-1}}{q - 1}, \, 0,\,0\right).\] Hence these first two entries will never be zero for $n>1$ and $q$ a prime power. So Theorem \ref{InnerDSpan} shows that all the point-pencils lie inside $V_0\perp V_1$. 
 
%Using Lemma \ref{FullRankIncidence}, we know that the point-pencils are linearly independent.
%Note that the point-pencils are linearly independent, since, due to Lemma \ref{FullRankIncidence}, we know that the point-line incidence matrix of AG($n,q$) has full rank.
 From Lemma \ref{Span} and the description of $Q$ in Equation (\ref{MatrixQ}), we obtain that \(\dim(V_0\perp V_1)=1+ \left( q^n-1\right)=q^n \). This number is equal to the number of point-pencils in $\AG(n,q)$. Together with Lemma \ref{FullRankIncidence}, we have that the point-pencils form a basis for the space \(V_0\perp V_1\).
\end{proof}

We now give a second result on these eigenspaces.
\begin{Le}\label{Allspans}
In the affine space AG($n,q$) with association scheme $\Delta$ (see Construction \ref{AssoScheme}), we have the following:
\begin{enumerate}
\item The line spreads of type II form a basis for the space $V_0\perp V_3$.
%\item The space $V_0 \perp V_2 \perp V_3$ \textbf{contains} all the line spreads of type $\text{III}^+$ and for the characteristic vector $\chi_\mathcal{S}$ of such a line spread, it holds that $E_2 \cdot \chi_\mathcal{S}\not = 0 \not = E_3 \cdot \chi_\mathcal{S}$.
\item The space $V_0 \perp V_2 \perp V_3$ is spanned by line spreads of type $\text{III}^+$  and for the characteristic vector $\chi_\mathcal{S}$ of such a line spread $\mathcal{S}$, it holds that $E_2 \cdot \chi_\mathcal{S}\not = 0 \not = E_3 \cdot \chi_\mathcal{S}$.
\end{enumerate}
\end{Le}
\begin{proof}
\begin{enumerate}
\item This is done in a similar way as the previous lemma. The inner distribution of a line spread \(\mathcal{S}_1\) of type II is equal to
 \[s_1=\left(1,\,0,\,q^{n - 1} - 1,\,0\right).\]
From this we obtain that \[s_1\cdot Q=\left(q^{n - 1},\,0,\,0,\,\frac{q^{2 n-1} - q^{n}}{q -
 	1}\right).\]
The first and last entry will never be zero for $n>1$ and $q$ a prime power.
 So from Theorem \ref{InnerDSpan}, we obtain that
 \(\chi_{\mathcal{S}_1}\in V_0\perp V_3\). Note that these line spreads are in fact subsets of point-pencils in the hyperplane at infinity in $\PG(n,q)$. But due to the fact that no two subsets contain the same line, we know that these line spreads are also linearly independent. From Lemma \ref{Span} and the description of $Q$ in Equation (\ref{MatrixQ}), we obtain that
	$$\dim(V_0\perp V_3)= 1+\frac{q^n-q}{q-1}= \frac{q^n-1}{q-1}.$$
	This dimension is equal to the number of spreads of type II, which proves that these line spreads form a basis.

\item Analogously for a line spread $\mathcal{S}_2$ of type $\text{III}^+$. The inner distribution is equal to
$$s_2=\left(1,\,0,\,q^{n - 2} - 1,\,q^{n - 1} - q^{n - 2}\right),$$
such that $$ s_2 \cdot Q =\left(q^{n - 1},\,0,\,q^{2n-2} - q^{n-2},\,-\frac{{\left(q^{2}
		- q + 1\right)} q^{n-2} - q^{2  n-2}}{q - 1}\right).$$ 
The first and third entry will never be zero for \(n>1\) and \(q\) a prime power. The last entry needs some arguments. If $(q^2-q+1)q^{n-2}-q^{2n-2}=0$, then $q=0$ or $q(q-1)=q^n-1$ and thus $q=0$ or $q^{n-1}+...+q^2+1=0$. This statement is never true if $n>1$ and $q$ a prime power. Hence, using Theorem \ref{InnerDSpan}, we obtain that $\chi_{\mathcal{S}_2}\in V_0\perp V_2\perp V_3$ and especially we have that $E_2 \cdot \chi_{\mathcal{S}_2}\not = 0 \not = E_3 \cdot \chi_{\mathcal{S}_2}$. 

To show that $V_0\perp V_2 \perp V_3$ is spanned by line spreads of type $\text{III}^+$, we use Theorem \ref{Conclusiontheorem0}. Suppose that these line spreads would span $V_0\perp W_1$, with $V_2 \perp V_3= W_1 \perp U_1$, then we want to show that $W_1=V_2 \perp V_3$. If there exists a $\psi \in U_1\setminus \{0\}$, then we know that $E_2 \cdot \psi \not = 0$ or $E_3 \cdot \psi \not = 0$. Let us now consider a line spread $\mathcal{S}$ of type $\text{III}^+$, then we know that its characteristic vector $\chi_\mathcal{S} \in V_0\perp W_1 \subseteq V_0\perp V_2 \perp V_3$. Hence $\chi_\mathcal{S}$ lies in the complementary space $V_0\perp W_1$ of $U_1$, thus $\chi_\mathcal{S}\cdot \psi=0$. Due to Lemma \ref{ThetaMaintainTypeIII}, we have that for every $\theta \in \AGL(n,q)$ it holds that $\chi_{\mathcal{S}^\theta}\cdot \psi=0$. So from Theorem \ref{Conclusiontheorem0}, we obtain that  $E_2 \cdot \chi_\mathcal{S}  = 0$ or $E_3 \cdot \chi_\mathcal{S}= 0$. This is a contradiction with the end of the preceding paragraph.
\end{enumerate}
\end{proof}

\subsection{The proof of Theorem \ref{MAIN}}
\begin{proof}
Consider the association scheme $\Delta$ from Construction \ref{AssoScheme} and let  $\mathcal{L}$ be a line set in $\AG(n,q)$ such that for every line spread $\mathcal{S}$ it holds that $|\mathcal{L}\cap \mathcal{S}|=x$. % be a Cameron-Liebler line class in $\AG(n,q)$ of parameter $x$. Denote the characteristic vector of $\mathcal{L}$ by $\chi_{\mathcal{L}}$. 
Then our goal is to prove that $\chi_{\mathcal{L}} \in V_0\perp V_1$, since, from Theorem \ref{pnt-pnc}, it then follows that $\chi_{\mathcal{L}} \in$ Im($A_n^T$)  and hence $\mathcal{L}$ is a Cameron-Liebler line class of parameter $x$. 

Consider $\mathcal{S}$ to be a line spread of type $\text{III}^+$. Such a line spread exists if we can choose $q$ different points in $\pi_{n-2}$. This is clearly the case if $n\geq 3$. If we denote the characteristic vector of $\mathcal{S}$ by $\chi_{\mathcal{S}}$, we know by the definition of $\mathcal{L}$ that
$$\chi_{\mathcal{L}}\cdot \chi_{\mathcal{S}}=x.$$
In combination with Lemma \ref{ThetaMaintainTypeIII}, we know that
$$\chi_{\mathcal{L}}\cdot \chi_{\mathcal{S}^\theta}=x,$$
for all $\theta \in \AGL(n,q)$. Hence, from Theorem \ref{Conclusiontheorem0} and Lemma \ref{Allspans} (Property (2)) which states that $E_2\cdot \chi_\mathcal{S} \not=0$ and that $E_3\cdot \chi_\mathcal{S} \not=0$, we may conclude that $E_2\cdot \chi_\mathcal{L}=0=E_3\cdot \chi_\mathcal{L}$. Thus using Theorem \ref{InnerDSpan}, we obtain that
$$\chi_{\mathcal{L}} \in V_0\perp V_1.$$
This proves the theorem.
\end{proof}

\section{Classification results}\label{sec:class}
In this section we will focus on some classification results of Cameron-Liebler $k$-sets in $\AG(n,q)$ with certain parameters. From now on we will use all the equivalent definitions of Theorem \ref{EquivalencesGeneral} to describe Cameron-Liebler $k$-sets. In order to obtain a classification result, we will need the following result.
\begin{St} \cite[Theorem 3]{Tanaka}\label{ErdosKoRadoSt}
	
	Let $0 \leq t \leq k$ be positive integers. Let \(\mathcal{S}\) be a set of $k$-spaces in $\PG(n,q)$, pairwise intersecting in at least a $t$-space. If $n\geq 2k+1$, then
	$$|\mathcal{S}|\leq \begin{bmatrix}
	n-t \\
	k-t
	\end{bmatrix}_q.$$
	Equality holds if and only if \(\mathcal{S}\) is the set of all $k$-spaces through a fixed $t$-space, or $n=2k+1$ and \(\mathcal{S}\) is the set of all $k$-spaces inside a fixed $(2k-t)$-space. %If $2k-t\leq n\leq 2k$, then 
	%$$|\mathcal{S}|\leq \begin{bmatrix}
	%2k+1-t \\
	%k-t
	%\end{bmatrix}_q.$$
	%Equality holds if and only if $\mathcal{S}$ is the set of $k$-spaces in a fixed $(2k-t)$-dimensional subspace.
\end{St}

Before we give the classification results, the reader should keep Example \ref{TrivialEx} in mind, where we gave some examples of Cameron-Liebler $k$-sets in $\PG(n,q)$. Note that by restriction to $\AG(n,q)$ we actually obtain fewer examples or stronger conditions on Cameron-Liebler $k$-sets. We first give the following lemma.
\begin{Le}\label{x=1Inn-1}
A non-empty set of $k$-spaces contained in a hyperplane of $\AG(n,q)$, is not a Cameron-Liebler $k$-set in AG($n,q$).
\end{Le}
\begin{proof}
Let $\mathcal{L}$ be a Cameron-Liebler $k$-set in $\AG(n,q)$, that consists out of a set of $k$-spaces inside a hyperplane $\pi$. Pick a $k$-space $K\in \mathcal{L}$, which we can consider in the projective closure PG($n,q$). Then we can define a type II $k$-spread $\mathcal{S}_1$ as the set of affine $k$-spaces through  $K\cap \pi_\infty$. Analogously we can define  $\mathcal{S}_2$ as the set of affine $k$-spaces through another $(k-1)$-space at infinity that does not lie in $\pi$. It is clear that
$$|\mathcal{L}\cap \mathcal{S}_1|\not= |\mathcal{L}\cap \mathcal{S}_2|=0.$$
This is a contradiction with Lemma \ref{2to8}.
\end{proof}
This lemma gives the following classification result.
\begin{St}\cite[Theorem 4.1]{BDF}
Let $q \in\{2,3,4,5\}$. Then all Cameron-Liebler $k$-sets in $\AG(n,q)$ consist out of all the $k$-spaces through a fixed point, if $k+1,n-k \geq 2$ and either (a) $n \geq 5$ or (b) $n = 4$ and $q = 2$.
\end{St}
\begin{proof}
Here we use the combination of Theorem \ref{CLkSpaceBasic_plus_lines}, Theorem \ref{Non-ExBD1} and Lemma \ref{x=1Inn-1}.
\end{proof}

\subsection{Cameron-Liebler $k$-sets with parameter $x=1$ in $\AG(n,q)$}
\begin{Ex}\label{Triv1AG}
Consider $\mathcal{L}$  as the set of $k$-spaces through a fixed affine point in $\AG(n,q)$. Then $\mathcal{L}$ is a Cameron-Liebler $k$-set in $\AG(n,q)$ of parameter $x=1$. This can be seen from  Corollary \ref{H6ProjToAff} together with the fact that $\mathcal{L}$ also is a Cameron-Liebler $k$-set in $\PG(n,q)$ of parameter $x=1$.
	
\end{Ex}
Using Theorem \ref{CharX=1PG(n,q)}, we know that for $n> 2k+1$ this example is the only example of a Cameron-Liebler $k$-set of parameter $x=1$  in PG($n,q$). If $n=2k+1$, the set of all $k$-spaces in a hyperplane also gives an example of a Cameron-Liebler $k$-set with parameter $x=1$. This fact gives the following theorem. %We can ask a similar question for AG($n,q$).
The following theorem also proves the first part of Theorem \ref{Smallx} from the introduction.
\begin{St}\label{Chap6Charac1}
	Consider the affine space $\AG(n,q)$ and let \(\mathcal{L}\) be a Cameron-Liebler $k$-set with parameter $x=1$ in this affine space. If also $n \geq 2k+1$, then \(\mathcal{L}\) consists of all the $k$-spaces through an affine point.
\end{St}

\begin{proof}
Using Theorem \ref{CLkSpaceBasic}, we obtain that every Cameron-Liebler $k$-set in $\AG(n,q)$ is a Cameron-Liebler $k$-set in $\PG(n,q)$. The latter will have the same parameter $x=1$. From Theorem \ref{CharX=1PG(n,q)} and Theorem \ref{x=1Inn-1}, the assertion follows.
%	By Lemma \ref{Charactx=1Kspaces}, we obtain that in the corresponding projective space $\PG(n,q)$, the $k$-spaces of \(\mathcal{L}\) pairwise intersect in at least a point. So we can use Theorem \ref{ErdosKoRadoSt} to obtain that
%	$$|\mathcal{L}|\leq \begin{bmatrix}
%	n \\
%	k
%	\end{bmatrix}_q.$$
%But due to {\color{red} the definition of a Cameron-Liebler  $k$-set} %due to Lemma \ref{SizeAffCLkSet}
%we in fact obtain equality. Using Theorem \ref{ErdosKoRadoSt} again, this then implies that $\mathcal{L}$ consists out of all the $k$-spaces through a point or in the specific case that $n=2k+1$, $\mathcal{L}$ could also consist out of all the $k$-spaces inside an $(n-1)$-space. Due to Lemma \ref{x=1Inn-1} this last case cannot occur.
\end{proof}
We also will be able to improve this result for $n\geq k+2$, see Corollary \ref{Un1}.

\subsection{Cameron-Liebler line classes of parameter $x=2$ in $\AG(n,q)$}
%Note that the following result is known.
%\begin{St}\cite[Corollary 4.5]{DMSS}\label{X=2AG(3,q)}
%There do not exist Cameron-Liebler line classes with parameter $x=2$ in AG($3,q$).
%\end{St}
% Our goal is to generalize this result to AG($n,q$). 
In this section, our goal will be to exclude the parameter $x=2$ for Cameron-Liebler line classes in $\AG(n,q)$, with $n\geq 3$. To do this, we will need the following lemma.
\begin{Le}\label{LinesIntersectionx=2}
	Consider an affine Cameron-Liebler line class $\mathcal{L}$ with parameter $x=2$ in $\AG(n,q)$, with $n\geq 4$. Then for every two points $p_1$ and $p_2$ in $\pi_\infty$, there are two lines of $\mathcal{L}$ through each of them. These $4$ lines generate at most a $3$-space.
\end{Le}
\begin{proof}
	Denote the lines of $\mathcal{L}$ through $p_1$ by $\ell_1$ and $\ell_2$, and denote the lines of $\mathcal{L}$ through $p_2$ by $r_1$ and $r_2$.
	\begin{figure}[h!]
		\includegraphics[scale=0.5]{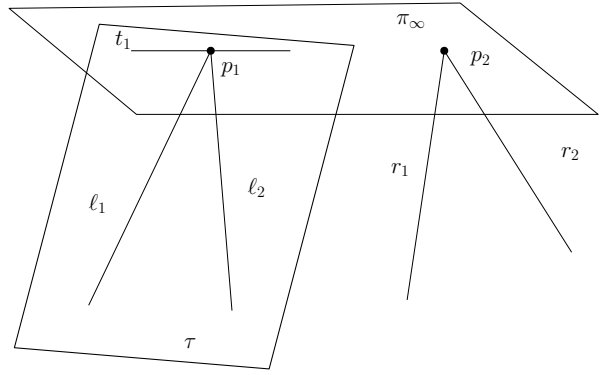}
		\centering
		\caption{Sketch for the proof of Lemma \ref{LinesIntersectionx=2}}\label{FigLinesIntersect}
		\centering
	\end{figure}
	We start by considering $\langle \ell_1,\ell_2,r_1 \rangle$, then we know that $\dim(\langle \ell_1,\ell_2,r_1 \rangle) \in \{2,3,4\}$. Suppose now that $\dim(\langle \ell_1,\ell_2,r_1 \rangle)=4$. If we now call the plane $\langle \ell_1,\ell_2\rangle=\tau$, then we know that 
	$\dim(\tau\cap \pi_\infty)=1.$
	This intersection line we call $t_1$, see Figure \ref{FigLinesIntersect}. Now we can find a $(n-2)$-space $\widetilde{\pi}\subseteq \pi_\infty$ such that 
	$$\langle t_1, p_2\rangle \subseteq \widetilde{\pi}$$
	and $$\langle \ell_1,\ell_2,r_1 \rangle\cap \pi_\infty \not \subseteq \widetilde{\pi}.$$
	This is possible, since first $\langle \ell_1,\ell_2,r_1 \rangle\cap \pi_\infty$ is a $3$-dimensional space that contains $\langle t_1, p_2\rangle$ as a $2$-dimensional subspace. And, secondly, since $n\geq 4$, we know that $n-2 \geq 2$.

 Thus with the identity of Grassmann, we obtain that
	$$\dim(\langle \widetilde{\pi}, \tau \rangle)=n-2+ 2-\dim( \widetilde{\pi}\cap \tau )=n-1,$$
$$\dim(\langle \widetilde{\pi},r_1\rangle)=n-2+1-\dim(\widetilde{\pi}\cap r_1)=n-1$$
	and
	\begin{equation*}
		\begin{split}
		\dim(\langle \ell_1, \ell_2, r_1, \widetilde{\pi} \rangle) &= \dim(\langle \ell_1,\ell_2, r_1 \rangle)+ \dim(\widetilde{\pi})- \dim(\langle \ell_1,\ell_2, r_1 \rangle \cap \widetilde{\pi})\\
&= \dim(\langle \ell_1,\ell_2, r_1 \rangle)+ \dim(\widetilde{\pi})- \dim((\langle \ell_1,\ell_2, r_1 \rangle\cap \pi_\infty) \cap \widetilde{\pi})\\
		&=4+(n-2)-2=n.
		\end{split}
	\end{equation*}
    Thus we can conclude that $\langle \widetilde{\pi}, \ell_1,\ell_2 \rangle \not= \langle\widetilde{\pi}, r_1 \rangle$. So we can define a line spread $\mathcal{S}$ of AG($n,q$) (of type\footnote{This exists due to $n\geq 4$.} III) that contains $\ell_1,\ell_2$ and $r_1$, such that $|\mathcal{L} \cap \mathcal{S}|\geq 3$, which is a contradiction. Thus from this we can conclude that 
	$\dim (\langle \ell_1,\ell_2,r_1\rangle)\leq 3.$ Analogously, we can obtain that $\langle \ell_1,\ell_2,r_2\rangle$ and in general every space generated by three of these four lines is at most a $3$-dimensional space. To show that these four lines span at most a $3$-space, we need to consider some cases.
	\begin{enumerate}
		\item First if $p_2 \not \in t_1$, then $\langle \ell_1,\ell_2,r_2\rangle$ intersects $\langle \ell_1,\ell_2,r_1\rangle$ in at least the point $p_2$ and the plane $\tau$. So, since $p_2\not \in \tau$, both $3$-spaces are the same. Hence, from now on, we assume that $p_2 \in t_1$.
		\item If $r_1$ and/or $r_2$ are contained in $\tau$, we are done, since these four lines span a plane or a $3$-space.
		\item If $r_1$ and $r_2$ are not contained in $\tau$ and $\langle r_1, r_2 \rangle \cap \tau= t_1$, then again we can conclude that all four lines lie in a $3$-space.
		\item If $r_1$ and $r_2$ are not contained in $\tau$ and $\langle r_1, r_2 \rangle \cap \tau\not= t_1$. Then $\langle r_1,r_2 \rangle\cap \pi_\infty\cap t_1=p_2$ and we analogously obtain from previous cases that  $\langle r_1,r_2,\ell_2\rangle$ and $\langle r_1,r_2,\ell_1\rangle$ are two $3$-spaces that now contain $p_1 \not\in \langle r_1,r_2 \rangle$ and $\langle r_1,r_2 \rangle$.
	\end{enumerate}
	 This proves the lemma.
\end{proof}

Let us now state the following known theorem.
\begin{St}[Folklore]\label{Folklore}
	Consider a set of $k$-spaces $\mathcal{E}$ in $\PG(n,q)$, $1\leq k\leq n-1$, such that every two $k$-spaces intersect in a $(k-1)$-space. Then $\mathcal{E}$ consists out of a subset of all the $k$-spaces through a fixed $(k-1)$-space or all the $k$-spaces inside a $(k+1)$-space.
\end{St}
We are ready to state the main theorem. This theorem proves a second part of Theorem \ref{Smallx} in the introduction.
\begin{St}\label{ConclParam2}%\label{Charactn>=4x=2}
	There does not exist a Cameron-Liebler line class $\mathcal{L}$ of parameter $x=2$ in $\AG(n,q)$, $n \geq 3$.
\end{St}
\begin{proof}
The case for $n=3$ is proven in \cite[Corollary 4.5]{DMSS}, so we may suppose that $n\geq 4$.\\
Suppose there exists a Cameron-Liebler line class $\mathcal{L}$ of parameter $x=2$. Then we can define $\mathcal{E}$ as the set of planes, such that each plane is defined by a point at infinity and the two corresponding lines of $\mathcal{L}$ through this point. Due to Lemma \ref{LinesIntersectionx=2} we know that these planes pairwise intersect in a line or coincide.
Using Theorem \ref{Folklore}, we can conclude that $\mathcal{E}$ consists out of a subset of all the planes through a fixed line or all the planes in a $3$-space $\sigma$. If $\mathcal{E}$ would consist out of all the planes in a $3$-space $\sigma$, then $\mathcal{L}$ is a set of lines inside $\sigma$ and thus inside a certain hyperplane. This is a contradiction with Lemma \ref{x=1Inn-1}. So we conclude that $\mathcal{E}$ consists out of planes through a fixed line $\ell$.
% first of all it is clear that $\sigma$ does not lie at infinity since $\mathcal{L}$ consists out of affine lines. So now we can choose a point $p$ at infinity and not in $\sigma$, this defines a line spread of type II, that will not contain a line of $\mathcal{L}$. This is a contradiction. So we conclude that $\mathcal{E}$ consists out of all the planes through a line $\ell$.

 If $\ell$ would be a line at infinity then $|\mathcal{L}|=2(q+1)$, since every point at infinity belongs to two lines of $\mathcal{L}$. Note that for $n\geq 3$ this number is strictly smaller than the size of a Cameron-Liebler line class in AG($n,q$) of parameter $x=2$, which has size $2\frac{q^n-1}{q-1}$. So the line $\ell$ should be affine. See Figure \ref{LastPosLinesF}.
\begin{figure}[h!]
	\includegraphics[scale=0.5]{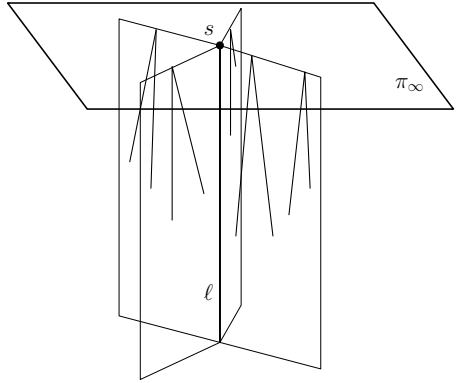}
	\centering
	\caption{The last possible option in AG($n,q$).}
	\label{LastPosLinesF}
	\centering
\end{figure}
%\item We now claim that
%$$|\mathcal{L}|\leq 2(q+1)|\mathcal{E}|	-2(|\mathcal{E}|-1).$$ (proof): For every plane in $\mathcal{E}$, we obtain at least two lines in $\mathcal{L}$. But we can also obtain that this plane contains at most $2(q+1)$ lines of $\mathcal{L}$. Since through every point at infinity there are at most $2$ lines of $\mathcal{L}$ (that lie in this plane). Secondly every time that we add an extra plane, this plane intersects the first added plane in a line, thus at infinity at a point. In this point we at least counted two lines through that point too much. This explains the second term.
%
%\item We now combine this fact with the observation that $\mathcal{E}\leq \frac{q^{n-1}-1}{q-1}$, which is the maximal number of planes through a line.
%Hence this will give that
%\begin{equation}\label{Char2}
%\begin{split}
%2\left(\frac{q^n-1}{q-1} \right) = |\mathcal{L}| & \leq 2(q+1)|\mathcal{E}|-2(|\mathcal{E}|-1) \\
%& \leq 2(q+1)\frac{q^{n-1}-1}{q-1}-2q \frac{q^{n-2}-1}{q-1} \\
%& \leq 2 \left( \frac{q^n-q+q^{n-1}-1-q^{n-1}+q}{q-1} \right) \\
%& \leq 2\left( \frac{q^n-1}{q-1} \right) 
%\end{split}
%\end{equation}
%Here we notice that we obtain an equation. Hence, $\mathcal{E}$ consists out of all the planes through $\ell$. Denote $s= \ell \cap \pi_\infty$.

Note that $|\mathcal{E}|= \frac{q^{n-1}-1}{q-1}$, since every point at infinity will define exactly one plane in $\mathcal{E}$ by definition. Hence this is also the number of all planes through a line, such that we know that $\mathcal{E}$ consists out of all the planes through $\ell$. Let us denote $s= \ell \cap \pi_\infty$.
Now we can use Theorem \ref{EquivalencesAG(n,q)lines}, where we have shown that being an affine Cameron-Liebler line class in $\AG(n,q)$ is equivalent with the following statement.
	For every affine line $\ell_1$, the number of affine lines in $\mathcal{L}$  disjoint to $\ell_1$ in $\AG(n,q)$ is equal to \begin{equation}\label{HEq} \left(q^2\frac{q^{n-2}-1}{q-1}+1\right)(x-\chi_{\mathcal{L}}(\ell_1)
	).\end{equation} Consider now an affine line $\ell'$ through $s$ that is contained in $\mathcal{L}$ and not equal to the intersection line $\ell$. Note that this is always possible, since $s$ belongs to exactly two lines of $\mathcal{L}$. Then all lines of $\mathcal{L}$ except those in the plane $\langle \ell,\ell'\rangle$, are disjoint to $\ell'$. Since every other plane through $\ell$ has $2q$ lines of $\mathcal{L}$ skew to $\ell'$ and we also need to count the other line of $\mathcal{L}$ through $s$,
this number is equal to 
	$$(|\mathcal{E}|-1)\cdot 2q+1.$$
With some calculations, we find that this equals
	$$2q^2\left( \frac{q^{n-2}-1}{q-1}\right) +1.$$
	This number should be equal to Equation (\ref{HEq}). In this equation we fill in $\chi_\cL(\ell')=1$, and we obtain that there should be  $\left(q^2 \frac{q^{n-2}-1}{q-1}+1\right) (2-1
	)$ lines disjoint to $\ell'\in \mathcal{L}$. These two numbers are not equal. So there does not exist Cameron-Liebler line classes with parameter $x=2$ in $\AG(n,q)$, $n\geq 4$ either.%, since otherwise it would follow that
%	\begin{equation*}
%		\begin{split}
%		2q^2\left( \frac{q^{n-2}-1}{q-1}\right) +1 &= q^2 \frac{q^{n-2}-1}{q-1}+1 \\
%		%\Leftrightarrow q^2\left( \frac{q^{n-2}-1}{q-1}\right)+1 &= 1 \\
%		\Leftrightarrow q^2 \left(  \frac{q^{n-2}-1}{q-1}\right) & =0. \\
%		\end{split}
%	\end{equation*}
%But this latter is false.
\end{proof}

%This theorem implies the following result.
%\begin{St}\label{ConclParam2}
%There do not exist Cameron-Liebler line classes in $\AG(n,q)$, $n\geq 3$, of parameter $x=2$.
%\end{St}
%\begin{proof}
%This is a combination of Theorem \ref{X=2AG(3,q)} and Theorem \ref{Charactn>=4x=2}.
%\end{proof}

\begin{Opm}\label{CharacInAG}
In contrast with the case for $x=1$, we have proven the preceding theorem without using Theorem \ref{CLkSpaceBasic}. The reason for this choice was that we are not able to deduce the preceding theorem from results in $\PG(n,q)$. Reducing to the projective case and using Theorem \ref{Non-ExBD1}, we observe that there do not exist Cameron-Liebler $k$-sets of parameter $x=2$ in $\AG(n,q)$ if $q\in \{2,3,4,5\}$ together with $n\geq 5$ or $n=4$, but $q=2$. 

While using Theorem \ref{JozefienClassific}, we obtain a non-existence result for $n\geq 5$ and $q\geq 3$. Both statements combined are weaker than the preceding theorem. This seems a small difference, but for general $k$ the results are even stronger. This will be proven in Corollary \ref{No2}.
%For example for Cameron-Liebler line classes in $\PG(n,q)$ this result was proven in Theorem \ref{Non-ExBD1} that for $q\in \{2,3,4,5\}$ together with $n\geq 5$ or $n=4$, but $q=2$.
%While in Theorem \ref{JozefienClassific}, it was proven for $n\geq 5$ and $q\geq 3$. So we also have given the case $n=3$ and $n=4$ which cannot be deduced from literature. For general $k$ it will be even more in our interest to not use results given in $\PG(n,q)$. This result is stated in Corollary \ref{No2}.
\end{Opm}

\subsection{Characterisation of the parameter $x$ of Cameron-Liebler $k$-sets in AG($n,q$)}

Our goal here is to prove that there do not exist Cameron-Liebler $k$-sets in $\AG(n,q)$ of parameter $x=2$, with $n\geq k+2$. As we already briefly discussed, we can not deduce this result so far by using Theorem \ref{CLkSpaceBasic} and thus reducing to the projective case. If we would reduce to $\PG(n,q)$ and use Theorem \ref{Non-ExBD1} and Theorem \ref{JozefienClassific}, we would obtain a non-existence result for (1) $n\geq 3k+2$ and $q\geq 3$, (2)  $n\geq 5$ with $q\in \{2,3,4,5\}$ and (3) $n=4$ with $q=2$.  Note that this statement is significantly weaker than the previous claim. In achieving this goal we will also obtain a minor non-existence condition on the parameters of certain Cameron-Liebler $k$-sets in AG($n,q$). We will do so by connecting every Cameron-Liebler $k$-set to a Cameron-Liebler line class of the same parameter. For this we will need the following observation.
\begin{Le}\label{AtInftyThrough}
	Let $\mathcal{L}$ be a Cameron-Liebler $k$-set with parameter $x$ in $\AG(n,q)$. Then the number of elements of \(\mathcal{L}\) through a fixed \(i\)-space at infinity, for $-1 \leq i \leq k-2$,  is equal to 
	\[\begin{bmatrix}
	n-i-1 \\
	k-i-1
	\end{bmatrix}_q x.\]
\end{Le}
\begin{proof}
Consider an $i$-space $I$ at infinity. Then we can count all the elements of $\mathcal{L}$ through $I$, by counting the number of $(k-1)$-spaces through $I$ inside $\pi_\infty$ and multiplying this by the number of elements of $\mathcal{L}$ through each $(k-1)$-space.  Both numbers are known, since through every $(k-1)$-space at infinity there are, by Lemma \ref{2to8}, in total $x$ elements of $\mathcal{L}$.
The assertion follows
\end{proof}
A remarkable observation is that $\gauss{n-i-1}{k-i-1}_q x$ equals the size of a Cameron-Liebler $(k-i-1)$-set in $\AG(n-i-1,q)$. This observation will lead to the following construction.
%Thus if we consider an arbitrary $i$-space $I$ at infinity, we find $\begin{bmatrix}
%	n-i-1 \\
%	k-i-1
%	\end{bmatrix}_q x$ (all affine) elements of $\mathcal{L}$ that contains this $i$-space. Here we remark that this is exactly the number of $(k-i-1)$-spaces of a Cameron-Liebler $(k-i-1)$-space of parameter $x$ in AG($n-i-1,q$). We can visually see these $(k-i-1)$-spaces by projecting the elements of $\mathcal{L}$ through $I$ on a $(n-i-1)$-space $\pi$ skew to $I$, see Figure 5.
\begin{Cons}\label{ConstructNew}
Consider $\mathcal{L}$ to be a Cameron-Liebler $k$-set in $\AG(n,q)$ of parameter $x$, and pick an $i$-space $I$ at infinity, for $0\leq i\leq k-2$ and $n\geq k+2$. Then, by Lemma \ref{AtInftyThrough}, there are \[\begin{bmatrix}
	n-i-1 \\
	k-i-1
	\end{bmatrix}_q x\]
$k$-spaces of $\mathcal{L}$ that contain $I$. Pick now an $(n-i-1)$-space $\pi$ in $\PG(n,q)$ skew to $I$. Then every $k$-space of $\mathcal{L}$ through $I$ will intersect $\pi$ in a $(k-i-1)$-space in $\PG(n,q)$, see Figure 5. 
\begin{figure}[h!]
\includegraphics[scale=0.4]{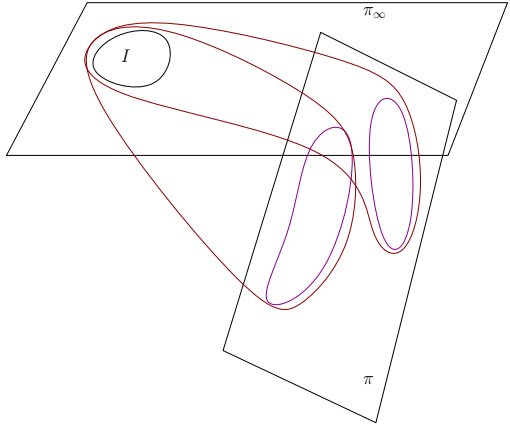}
\centering
\caption{The red elements are $k$-spaces in $\mathcal{L}$ through $I$ and we get the purple $(k-i-1)$-spaces after the projection onto the  $(n-i-1)$-space $\pi$.}
\centering
\end{figure}
We shall denote this set of $(k-i-1)$-spaces by $\mathcal{J}$. Remark that $\pi$ is in fact a projective space of dimension $n-i-1$, which is an affine space $\pi_A$ ($\simeq \AG(n-i-1,q)$) of the same dimension with hyperplane at infinity equal to $\pi\cap\pi_\infty$.

\end{Cons}

We first give a result on $k$-spreads.
\begin{Le}\label{SpreadMagic}
Consider Construction \ref{ConstructNew}. Then every $(k-i-1)$-spread in $\pi_A$  can be extended to an affine $k$-spread in $\AG(n,q)$, such that all the $k$-spaces in this $k$-spread contain $I$.
\end{Le}
\begin{proof}
Let $\mathcal{S}'$ be a $(k-i-1)$-spread in $\pi_A$, then  
$$\mathcal{S}:=\{ \langle I, N \rangle \mid N \in \mathcal{S}'\}$$
is a $k$-spread in AG($n,q$).
\end{proof}

\begin{St}\label{CLkSetReduction}
Suppose that $n\geq 2k-i$. Consider Construction \ref{ConstructNew}. Then the set $\mathcal{J}$ is a Cameron-Liebler $(k-i-1)$-spaces in $\pi_A$ ($\simeq \AG(n-i-1,q)$), which has the same parameter $x$.
\end{St}
\begin{proof}
Consider $\mathcal{L}$ and $\mathcal{J}$ as in Construction \ref{ConstructNew}, then we need to prove that $\mathcal{J}$ is a Cameron-Liebler $(k-i-1)$-set with the same parameter $x$ in $\pi_A$. We know that, due to Lemma \ref{SpreadMagic}, every $(k-i-1)$-spread $\mathcal{S}'$ in $\pi_A$ can be extended to a $k$-spread $\mathcal{S}$ in $\AG(n,q)$ such that every element of $\mathcal{S}$ contains $I$. Since $\mathcal{L}$ is a Cameron-Liebler $k$-set in $\AG(n,q)$, it holds that $|\mathcal{L}\cap\mathcal{S}|=x$. But since every $k$-space of $\mathcal{L}\cap\mathcal{S}$ contains $I$, it projects to an element of $\mathcal{J}\cap\mathcal{S}'$ and vice versa. So we have that $|\mathcal{J}\cap\mathcal{S}'|=x$. Using Theorem \ref{EquivalencesGeneral}, which gives the condition on $n$, we have proven the theorem.
%and define its projection on $\pi$ (through $I$) by $\mathcal{L}'$. Then we need to prove that $\mathcal{L}'$ is a Cameron-Liebler $(k-i-1)$-space with the same parameter $x$ in AG($n-i-1,q$). We know that, due to Lemma \ref{SpreadMagic}, every affine $(k-i-1)$-spread $\mathcal{S}'$ in $\pi$ can be extended to a $k$-spread $\mathcal{S}$ in AG($n,q$). Since $\mathcal{L}$ is a Cameron-Liebler $k$-set in AG($n,q$), it holds that $|\mathcal{L}\cap\mathcal{S}|=x$. This in combination with the fact that all elements\footnote{This implies that all the elements of $\mathcal{L}\cap\mathcal{S}$ are seen by the projection in $\pi$.} of $\mathcal{S}$ go through $I$, we have that $|\mathcal{L}'\cap\mathcal{S}'|=x$. This proves the theorem.
\end{proof}
\begin{Opm}
Note that this construction cannot be done in a similar way for $\PG(n,q)$.
\end{Opm}
We now have found a way to reduce Cameron-Liebler $k$-sets to Cameron-Liebler line classes of the same parameter $x$. Hence, this will lead to a transfer of non-existence results.
\begin{St}\label{ToLines}
Let $\mathcal{L}$ be a Cameron-Liebler $k$-set in AG($n,q$), with $n\geq k+2$. Suppose now that $\mathcal{L}$ has parameter $x$, then $x$ satisfies every condition which holds for Cameron-Liebler line classes in AG($n-k+1,q$).
\end{St}
\begin{proof}
We can use Theorem \ref{CLkSetReduction} for $i=k-2$, thus $n\geq k+2$, and obtain that there exists a Cameron-Liebler line class in $\AG(n-(k-2)-1,q)$ with the same parameter $x$. This proves the theorem.
\end{proof}
This theorem has the following consequences.
\begin{St}%(\cite[Theorem 1.1]{MetschAndGavrilyuk} and \cite[Corollary 4.3]{DMSS})
	Suppose that \(\mathcal{L}\) is a Cameron-Liebler $(n-2)$-set with parameter \(x\) of $\AG(n,q)$, then it holds that
	\begin{equation*}
	\frac{x(x-1)}{2} \equiv 0 \mod (q+1).
	\end{equation*}
\end{St} 
\begin{proof}
Use Theorem \ref{ToLines} for $n-2=k$ and the modular equality from  \cite[Corollary 4.3]{DMSS} or \cite[Theorem 1.1]{MetschAndGavrilyuk}.
\end{proof}
The following corollary completes the proof of $x=2$ in Theorem \ref{Smallx} in the introduction.
\begin{Gev}\label{No2}
There do not exist Cameron-Liebler $k$-sets in $\AG(n,q)$ with parameter $x=2$, with $n\geq k+2$.
\end{Gev}
\begin{proof}
If there would exist a Cameron-Liebler $k$-set $\mathcal{L}$ of parameter $x=2$ in $\AG(n,q)$, we can use Theorem \ref{ToLines} and obtain that there exists a Cameron-Liebler line class in $\AG(n-k+1,q)$ with parameter $x=2$. Since $n-k+1\geq 3$, we may use Theorem \ref{ConclParam2} to obtain a contradiction.
\end{proof}

We also have the following improvement of Theorem \ref{Chap6Charac1} To conclude Theorem \ref{Smallx}.
\begin{Gev}\label{Un1}
The only Cameron-Liebler $k$-set $\cL$ of parameter $x=1$ in $\AG(n,q)$, with $n\geq k+2$, consists of the set of $k$-spaces through a fixed point.
\end{Gev}
\begin{proof}
Again we use Theorem \ref{ToLines} to obtain a Cameron-Liebler line class $\cL'$ of parameter $x=1$ in $\AG(n-k+1,q)$. But combined this with Theorem \ref{Chap6Charac1} for $k=1$ in $\AG(n-k+1,q)$, we obtain that $\cL'$ consists out of all lines through a fixed point. Using Construction \ref{ConstructNew}, it is easy to see that $\cL$ is the set of all $k$-spaces through a point.
\end{proof}

%\begin{proof}
%If there would exist a Cameron-Liebler $k$-set in $\AG(n,q)$, we can use Theorem \ref{CLkSetReduction} for $i=k-2$ and obtain that there exists a Cameron-Liebler line class in $\AG(n-(k-2)-1,q)$ of parameter $x=2$. Due to Theorem \ref{ConclParam2} this is false for $n-k+1\geq 3$ and thus for $n\geq k+2$.
%\end{proof}
\section{Cameron-Liebler sets of hyperplanes in $\AG(n,q)$}\label{sec:hyp}
In this section we study Cameron-Liebler sets of hyperplanes in $\AG(n,q)$. We will be able to give a complete classification. This will be done by giving a classification of affine $(n-1)$-spreads.
\begin{Le}\label{Hyperplanespreads}
The only $(n-1)$-spreads in $\AG(n,q)$ are spreads of type II.
\end{Le}
\begin{proof}
Let $\mathcal{S}$ be an $(n-1)$-spread, then we need to prove that $\mathcal{L}$ is of type II. Consider now the projective closure $\PG(n,q)$, then we know that every two hyperplanes of $\mathcal{S}$ will intersect in an $(n-2)$-space. Due to the fact that $\mathcal{S}$ is an affine $(n-1)$-spread, these intersections must lie at infinity. But since every affine hyperplane only has an $(n-2)$-dimensional intersection with infinity, all these $(n-2)$-spaces need to be the same.
Hence, we have that $\mathcal{S}$ is of type II.
\end{proof}
The fact that we are able to classify all $(n-1)$-spreads in $\AG(n,q)$, gives us information how we can construct these Cameron-Liebler sets of hyperplanes in $\AG(n,q)$.
\begin{St}
Let $\mathcal{L}$ be a  set of affine hyperplanes  in $\AG(n,q)$ and consider the projective closure $\PG(n,q)$. Then $\mathcal{L}$ is a Cameron-Liebler set of hyperplanes of parameter $x$ if and only if $\mathcal{L}$ is a set of hyperplanes such that through every $(n-2)$-space at infinity we have chosen $x$ arbitrary hyperplanes.
\end{St}
\begin{proof}
The proof is similar as we have done for lines in Section 5. Hence, we find a $2$-class association scheme $\Delta=(\Phi_{n-1}, \mathcal{R})$, with $\mathcal{R}:=\{\mathcal{R}_0, \mathcal{R}_1, \mathcal{R}_2\}$. Here $\mathcal{R}_0$ is the identical relation, while $\mathcal{R}_1$ is the relation that denotes that the two hyperplanes are disjoint and thus intersect in an $(n-2)$-space at infinity and $\mathcal{R}_2$ is the relation that denotes intersection.  Using similar techniques as in Section 5, we obtain that 
$$P=\begin{pmatrix} 1 &  q-1 &\frac{q^{n+1}-1}{q-1}  \\  1 & q-1 & -1\\ 1 & -1 & 0 \\\end{pmatrix} \text{ and } Q=\begin{pmatrix} 1  &\frac{q^{n}-q}{q-1}&  q^{n}-1\\
1  &\frac{q^{n}-q}{q-1}&  -\frac{q^{n}-1}{q-1}\\
1 & -1 & 0\end{pmatrix}.$$
Due to Lemma \ref{Hyperplanespreads}, we have that the inner distribution of an $(n-2)$-spread is equal to the vector $v=(1,q-1,0)$. Also point-pencils have inner distribution $w=(1, 0, \frac{q^n-1}{q-1}-1)$. Using Theorem \ref{InnerDSpan}, we know that the $(n-2)$-spreads lie inside $V_0\perp V_1$ and the point-pencils lie inside $V_0\perp V_2$, in fact both sets will span these spaces. Combining this with Theorem \ref{Conclusiontheorem0}, we find that a Cameron-Liebler $(n-2)$-set can be characterized by the constant intersection of $(n-2)$-spreads. Hence, due to  Lemma \ref{Hyperplanespreads}, the assertion follows.
\end{proof}

%\section*{7 \,\,\, Cameron-Liebler sets of hyperplanes in $\AG(n,q)$ (bis)}
%
%We give a complete classification of Cameron-Liebler sets of hyperplanes in
%$\AG(n, q)$. For this, note that the graph $\Gamma$ with hyperplanes of $\AG(n, q)$
%with two hyperplanes adjacent if their intersection is non-empty, is the
%disjoint union of $m$ complete graphs of some constante size $c$ (in our case,
%$m=(q^n-1)/(q-1)$ and $c=q$). Clearly, $\Gamma$ has eigenvalues $q-1$ and $-1$ and its
%equitable bipartitions are well known. It is straightforward to see that a
%Cameron-Liebler set of hyperplanes corresponds to an equitable bipartition
%with eigenvalue $-1$. Hence, we obtain the following.
%\begin{St}
%Let $\mathcal{L}$ be a  set of affine hyperplanes  in $\AG(n,q)$ and consider the projective closure $\PG(n,q)$. Then $\mathcal{L}$ is a Cameron-Liebler set of hyperplanes of parameter $x$ if and only if $\mathcal{L}$ is a set of hyperplanes such that through every $(n-2)$-space at infinity we have chosen $x$ arbitrary hyperplanes.
%\end{St}

\section{Future research} \label{sec:fut}
We want to end this paper with some suggestions for further research. %There is still a lot of work to be done to obtain similar results for Cameron-Liebler $k$-sets in $\AG(n,q)$  for $k$ and $n$ without any restrictions.  This can be done by generalizing Section 6, which means a generalization of the denoted association scheme. This will lead to a generalization of the affine association scheme, which can be interesting to calculate the eigenvalues for.
 One could also attempt to classify more parameters of a Cameron-Liebler $k$-set in $\AG(n,q)$, since intuitively this will be less difficult than $\PG(n,q)$. We remind the reader of Theorem \ref{CLkSpaceBasic_plus_lines}, which states that Cameron-Liebler $k$-sets in $\AG(n,q)$ are special cases of Cameron-Liebler $k$-sets in $\PG(n,q)$. Another interesting problem  is to look for examples of Cameron-Liebler $k$-sets in $\AG(n,q)$. Note that it would be enough to find a Cameron-Liebler $k$-set in $\PG(n,q)$ that does not contain any $k$-spaces inside a hyperplane, see Theorem \ref{H6ProjToAff}.
\paragraph{Acknowledgement}
The research of Jozefien D'haeseleer and Ferdinand Ihringer is supported by the FWO (Research Foundation Flanders).

The authors thank the referees for their suggestions to improve this article.

%\nocite{*}
%\bibliographystyle{plain}
%\bibliography{sourcesArtikel}

\end{document}